\documentclass[a4paper,12pt]{article}

\advance\textheight2.9cm        
\advance\topmargin-2.0cm
\advance\textwidth2.6cm
\advance\evensidemargin-1.6cm

\usepackage{amsfonts} 
\usepackage{graphicx}
\usepackage{subcaption}
\usepackage{url}
\usepackage{cite}

\usepackage{theorem}
\usepackage{amsfonts}
\usepackage{amssymb} 
\usepackage{amsmath}
\usepackage{longtable}
\usepackage{array}
\usepackage{bm}  
\def\comment#1{{}}	

\usepackage{theorem}
\newtheorem{thm}{Theorem}
\newtheorem{lem}{Lemma}

\newtheorem{cor}{Corollary}
\newtheorem{prop}{Proposition}
\newtheorem{conj}{Conjecture}
\newtheorem{rem}{Remark}

\usepackage{bm} 
\def\conv{\fct{conv}}
\def\cl{\fct{cl}}

\def\median{\fct{median}}

\def\CG{\texttt{CG}}

\def\MCS{\texttt{MCS}}

\def\Box{{\hbox{\raisebox{0.0em}{\rlap{$\sqcap$}}\kern0em%
            \raisebox{-0.0em}{$\sqcup$}}} } 

\newenvironment{proof}{{\it Proof. }}{\nopagebreak\hspace*{0.5cm}\hfill$\Box$\vspace{0.5cm}}

\def\bepf{\begin{proof}}
\def\epf{\end{proof}}

\def\bfi#1{\textbf{#1}} 

\def\eeq{\end{equation}}
\def\lbeq#1{\begin{equation} \label{#1}}
\def\bary{\begin{array}}
\def\eary{\end{array}}

\def\D{\displaystyle}				
\def\ol{\overline}

\def\wt{\widetilde}

\def\eps{\varepsilon}

\def\Nz{\mathbb{N}}

\def\Rz{\mathbb{R}}

\def\ve {{\bf e}}

\def\bp {{\bf p}}

\def\argmin{\fct{argmin}}
\def\argmax{\fct{argmax}}

\def\dist{\fct{dist}}

\def\fct#1{\mathop{\rm #1}}	                

\def\median{\fct{median}}

\def\opt{\fct{opt}}

\def\D{\displaystyle}				
\def\ol{\overline}

\def\ealg{\end{alg}}

\usepackage{array,multirow,graphicx}
\usepackage[utf8]{inputenc}

\usepackage{graphicx}

\usepackage{colortbl}

\usepackage[table]{xcolor}

\usepackage{fancyhdr}

\usepackage{hyperref}
\usepackage{cleveref}

\usepackage{array,multirow,graphicx}

\usepackage{etex}

\usepackage{longtable}
\usepackage{lscape}
\usepackage{hhline}
\usepackage{caption}
\usepackage{blindtext}
\usepackage{lmodern}
\usepackage{textcomp}
\usepackage{chngcntr}
\usepackage{tikz}
\usepackage{csquotes}
\usepackage[english]{babel}
\usepackage[utf8]{inputenc}
\usepackage{algorithm}
\usepackage{xcolor}
\usepackage{ulem}
\usepackage{cancel}
\usepackage{multicol}
\usepackage{algpseudocode}

\fancypagestyle{plain}{\fancyhf{}\fancyfoot[C]{\bfseries \thepage}}

\newcolumntype{?}{!{\vrule width 1pt}}

\usepackage{adjustbox}

\newcounter{subsubsubsection}[subsubsection]
\def\subsubsubsectionmark#1{}

\def\subsubsubsection{\@startsection
	{subsubsubsection}{4}{\z@} {-3.25ex plus -1
		ex minus -.2ex}{1.5ex plus .2ex}{\normalsize\bf}}
\def\l@subsubsubsection{\@dottedtocline{4}{4.8em}
	{4.2em}}

\newcolumntype{R}[2]{%
	>{\adjustbox{angle=#1,lap=\width-(#2)}\bgroup}%
	l%
	<{\egroup}%
}

\def\bfi#1{{\bf{#1}}}

\usetikzlibrary{matrix,shapes,arrows,positioning,chains}

\def\bfi#1{{\bf{#1}}}



\usepackage{comment}
\usepackage{tikz}
\usetikzlibrary{matrix,shapes,arrows,positioning,chains}

\definecolor{ao(english)}{rgb}{0.0, 0.5, 0.0}

 \numberwithin{equation}{section}
 
 \newcommand{\R}{\mathbb{R}}

 \newcommand{\be}{\beta}

 \newcommand{\bd}{\mathbf{d}}
 \newcommand{\bx}{\mathbf{x}} 
  
 \newcommand{\bz}{\mathbf{z}}

 \newcommand{\bg}{\mathbf{g}}
 \newcommand{\by}{\mathbf{y}}
\newcommand{\bs}{\mathbf{s}}
\newcommand{\bv}{\mathbf{v}}
\newcommand{\bu}{\mathbf{u}}
\newcommand{\bw}{\mathbf{w}}
 \newcommand{\bq}{\mathbf{q}}

\begin{document}

\begin{center}

{\Large \bf An Approximate Conjugate Subgradient Algorithm with Matrix Parameter for Derivative-Free Nonsmooth Optimization Problems}

\vspace{0.5cm}

{\large \bf Morteza Kimiaei}
\centerline{\sl Fakult\"at f\"ur Mathematik, Universit\"at Wien}
\centerline{\sl Oskar--Morgenstern--Platz 1, A--1090 Wien, Austria}
\centerline{\sl email: kimiaeim83@univie.ac.at}
\centerline{\sl WWW: \url{http://www.mat.univie.ac.at/~kimiaei}}

\vspace{0.5cm}

{\large \bf Saman Babaie--Kafaki}
\centerline{\sl Faculty of Engineering, Free University of Bozen--Bolzano}
\centerline{\sl NOI Techpark, Via Bruno Buozzi 1, 39100 Bolzano (BZ), Italy}
\centerline{\sl email: saman.babaiekafaki@unibz.it}

\vspace{0.5cm}

{\large \bf Zohre Aminifard}
\centerline{\sl Université Catholique de Louvain (UCLouvain)}
\centerline{\sl Institute of Information and Communication Technologies, Electronics and} 
\centerline{\sl  Applied Mathematics, Place du Levant 3, B--1348, Louvain--la--Neuve, Belgium}
\centerline{\sl email: zohreh.aminifard@uclouvain.be}

\vspace{0.5cm}

\end{center}

{\bf Abstract.} We propose a derivative-free matrix conjugate-subgradient method for unconstrained nonsmooth optimization of locally Lipschitz functions. The method constructs discrete gradients using only function values and forms a finite sampled model of the Goldstein subdifferential. A minimal-norm element of the convex hull of the sampled discrete gradients is then computed and used both as a stationarity measure and as the reference vector for generating descent-oriented directions. To improve robustness beyond the basic steepest-descent direction, we introduce a matrix memory correction together with coefficient damping, diagonal scaling, bounded-angle correction, and matrix-stability safeguards. A two-point line-search procedure with enrichment is used to obtain either a serious step or an improved local model. Under suitable consistency assumptions on the discrete-gradient approximation and line-search sampling, the method generates directions satisfying a safeguarded descent property and computes approximate Goldstein stationary points. Numerical experiments on nonsmooth test problems with dimensions up to \(1000\)
show that both proposed variants are robust for lower and medium accuracy
requirements, while the matrix conjugate-subgradient variant remains the most
reliable under the strictest tolerance.


\vspace{0.2cm} {\em 2000 AMS Subject Classification: 90C53, 65K05.}

\hfill \today


\section{Introduction}\label{Introduction}

In the modern age, nonlinear programming plays a central role in various practical aspects of machine learning and data mining \cite{Elden}. Therefore, it deserves to be addressed thoroughly from both algorithmic and modeling perspectives. In recent years, scholars have devoted a substantial portion of their efforts to enriching optimization tools—particularly by emphasizing diversity and inclusion within the appropriate context. While these contributions can be considered satisfactory, the ever-expanding horizons of science and technology call for a broader and more adaptive perspective.

By the beginning of the current century, the growing use of high-dimensional models had led researchers to place increasing emphasis on scalable and computationally efficient algorithmic strategies. This development has also affected how accuracy is treated, since high accuracy requirements often increase the computational burden and may make it difficult to obtain an acceptable solution within a reasonable CPU time. Therefore, a central issue in modeling and optimization is to achieve a meaningful trade-off between accuracy and efficiency, which requires the use of algorithmic strategies specifically designed to control computational cost without sacrificing the quality of the computed solution.

In the context of continuous optimization, \textbf{Conjugate Gradient} (\texttt{CG}) algorithms have traditionally been recognized as a family of memoryless methods designed to address large-scale models \cite{Andreibook}, especially following recent hardware developments. As \textbf{Line Search} (\texttt{LiS}) techniques, they not only directly exploit the steepest descent direction as part of their search direction, but also indirectly incorporate Hessian information—making them often excellent in terms of global convergence and ease of implementation \cite{CGConvDai}. In some cases, \texttt{CG} methods have been endowed with an additional degree of flexibility by incorporating an extra parameter into their formulaic structure. Such extensions are designed to establish connections with other classes of optimization algorithms, such as \textbf{Quasi–Newton} (\texttt{QN}) methods, which more explicitly exploit second-order information of the model. An illustrative example in this context is the \textbf{Dai–Liao} (\texttt{DL}) \cite{DaiLiaoNCG} algorithm, which can nowadays be regarded as a sophisticated memoryless optimization approach \cite{SBKRairo}. Meanwhile, such algorithms have often been devised based on scalar parameters, thereby only benefiting from a so-called one-dimensional flexibility.


\textbf{Derivative-Free Nonsmooth Optimization} (DFNO) in algorithmic point of view can be broadly divided into two categories:
\begin{itemize}
    \item[$-$] An algorithm of the first class attempts to approximate generalized first-order information from function values. Discrete-gradient methods approximate subgradient-like objects through finite-difference sampling and then use these approximations inside \texttt{LiS} frameworks \cite{Bagirov2007,Karmitsa2012}. Simplex-gradient techniques construct local linear models from sampled function values and use the resulting gradient estimates mainly as descent indicators, for example to order polling directions in \textbf{Generalized Pattern Search} (GPS) or \textbf{Mesh Adaptive Direct Search} (MADS) \cite{Custodio2008}. Related approaches include clustering-based approximations of Clarke generalized gradients \cite{Gaudioso2024}, derivative-free variants inspired by gradient-sampling approaches \cite{Kiwiel2010,Hare2013,Hare2013a}, and max-linear trust-region models designed to capture local nonsmooth structure through collections of linear pieces \cite{Liuzzi2019}. These methods can improve practical efficiency, but their reliability depends on the quality, stability, and cost of the approximate first-order information.

    \item[$-$] An algorithm of the second class avoids explicit subgradient approximations and instead obtains convergence from polling geometry, dense direction sets, sufficient decrease, or smoothing mechanisms. MADS is the canonical example: it uses mesh-based polling with asymptotically dense directions to obtain Clarke-type stationarity guarantees without estimating subgradients \cite{audet2006mesh}. The {\tt NOMAD} software package implements MADS and incorporates additional search strategies, including \textbf{Variable Neighborhood Search} (VNS), to improve global exploration while preserving the underlying MADS convergence guarantees \cite{Audet2007}. \texttt{LiS}-based {\tt DFN} methods use dense direction sequences, sufficient-decrease tests, and exact-penalty mechanisms for nonsmooth constrained problems \cite{FasLLR14}. The {\tt DFNDFL} solver, proposed by Giovannelli et al.~\cite{Giovannelli2022}, extends the \texttt{LiS}-based {\tt DFN} framework to mixed-integer nonsmooth constrained optimization by combining dense direction sequences for continuous variables with primitive directions for integer variables and by employing an exact-penalty approach to handle nonlinear constraints. Other relevant derivative-free frameworks include rigorous pattern-search methods, Lipschitzian global optimization methods such as {\tt DIRECT}, and brute-force derivative-free solvers \cite{booker1999rigorous,torczon1997convergence,abramson2006convergence,finkel2004convergence,jones1993lipschitzian,PorT}. 
\end{itemize}
In our approach, the canonical direction set will be enriched by blending deterministic hypercube directions with randomized orthogonal bases \cite{Custodio2008,Audet2007}. The conditioning of the direction system will be monitored, and directions will be regenerated when necessary, improving finite-precision stability and angular coverage for efficient DFNO algorithms.

\subsection{Our Contribution}

This work develops a derivative-free matrix conjugate-subgradient framework for
locally Lipschitz nonsmooth optimization. The main algorithmic idea is to combine
discrete-gradient approximations of Goldstein subdifferentials with a
matrix-based memory mechanism for generating search directions. In contrast to
classical scalar memory corrections, the proposed approach uses a structured
matrix correction to incorporate information from recent iterations while
remaining suitable for function-value-only optimization. This matrix formulation
provides additional flexibility in the construction of search directions and
allows the use of stability safeguards based on matrix quantities, such as
eigenvalue and conditioning information.

The main contributions of this work are as follows:
\begin{itemize}

\item We develop a derivative-free framework for locally Lipschitz nonsmooth
optimization based on discrete gradients and finite sampled approximations of
Goldstein subdifferentials. The method uses only function values and does not
require analytical gradients, subgradients, or active-index information.

\item We compute a minimal-norm element of the convex hull of sampled discrete
gradients and use it as a practical stationarity residual and as the reference
vector for constructing descent-oriented search directions.

\item We introduce a matrix conjugate-subgradient direction that enriches the
basic discrete-gradient steepest-descent direction by incorporating memory from
previous iterations. The direction is protected by coefficient damping,
dominance control, diagonal scaling, matrix-stability checks, and a bounded-angle
safeguard.

\item We design a two-point line-search and enrichment procedure that either
accepts a serious step or augments the local discrete-gradient model. This
mechanism improves the sampled model while preserving a function-value-only
implementation.

\item We establish the consistency of the finite discrete-gradient model with
Goldstein-type stationarity under appropriate approximation assumptions and prove
descent and finite-termination properties for the safeguarded line-search
procedure.

\item We prove that, under the stated consistency assumptions and vanishing
sampling and discretization errors, the proposed framework generates approximate
Goldstein stationary points whose accumulation points are Clarke stationary.

\item We evaluate the proposed methods against established derivative-free
nonsmooth solvers on test problems with dimensions up to \(1000\). The numerical
results show that both proposed variants are robust for lower and medium
accuracy requirements, while the matrix conjugate-subgradient variant remains
the most reliable under the strictest tolerance.

\end{itemize}

In summary, the contribution has two components. First, we adapt the
exact-subgradient sampling and enrichment framework proposed in \cite{Maleknia2024} to the derivative-free
case by replacing exact subgradients with discrete-gradient approximations
computed from function values. This requires rebuilding the sampled model,
the minimal-norm stationarity residual, and the enrichment mechanism in terms
of approximate subgradient information. Second, we introduce new algorithmic
components that are not present in the exact-subgradient framework: a
derivative-free matrix conjugate-subgradient direction, a componentwise
diagonal scaling based on previously found best points, and safeguards based
on damping, dominance control, matrix stability, and bounded-angle correction.
The line search is also modified for robustness by using a two-point acceptance
test rather than relying only on the standard geometric update. These changes
allow the method to exploit memory information while remaining fully
derivative-free and stable under approximate subgradient models.

\subsection{Organization of the Study}

The organization of our study is summarized as follows. 
Section~\ref{sec:preliminaries} introduces the problem setting, stationarity
concepts, discrete-gradient approximations, and the finite sampled
subdifferential model. Section~\ref{MainCSG} presents the proposed matrix
conjugate-subgradient direction and the associated safeguard mechanisms.
Section~\ref{subsecTPLS} develops the two-point line-search and enrichment
procedure. Section~\ref{sec_stationarity_algorithm} gives the complete
algorithm and its convergence properties. Section~\ref{sec:numerical_results}
reports the numerical experiments and compares the proposed variants with
benchmark DFON solvers. Finally,
Section~\ref{sec:conclusion} concludes the paper.

\section{Preliminaries}
\label{sec:preliminaries}

This section collects the basic concepts and notation used in the development of
the proposed method. We first state the nonsmooth unconstrained optimization
problem, then recall Clarke and Goldstein stationarity, and finally introduce
the discrete-gradient constructions and finite sampled models used to obtain
computable derivative-free stationarity measures.

\subsection{Problem Formulation}
\label{subsec:problem_setting}

We consider the unconstrained minimization problem 
\begin{equation}\label{UP}
 \min_{\bx\in \mathbb{R}^n}\ f(\bx),
\end{equation}
where $f:\mathbb{R}^n \to \mathbb{R}$ is assumed to be locally Lipschitz, possibly nonconvex and nonsmooth, contributing significantly to practical implementations in machine learning and data mining \cite{Elden}. In particular, a wide range of constrained optimization algorithms need to locally address the unconstrained transformation of the original model within the framework of a subproblem, which highlights the importance of model \eqref{UP}.  In many modern applications, the objective function $f$ in \eqref{UP} is locally Lipschitz but may fail to be continuously differentiable. In such cases, classical gradient-based algorithms cannot be applied directly, since gradients may be unavailable or not exist everywhere. 

\subsection{Stationarity Concepts and Practical Criterion}
\label{subsec:stationarity}

In DFNO, it is well-known that the use of
approximate gradients, typically obtained via finite differences, 
can significantly improve performance, provided that the noise level
is sufficiently controlled. A similar principle applies in the
nonsmooth setting: approximations of subgradients constructed via
finite-difference schemes, such as discrete gradients
\cite{Bagirov2007,Karmitsa2012}, can lead to more effective descent
directions than purely directional sampling strategies. The following definitions of Clarke stationarity and $\varepsilon$-Clarke stationarity are standard in nonsmooth optimization and are based on the Clarke generalized directional derivative and Clarke subdifferential; see, e.g., \cite{audet2006mesh} and \cite{FasLLR14}.

\paragraph{Clarke stationarity.}
A point $\bx \in \R^n$ is called a \bfi{Clarke stationary point} of a locally Lipschitz function $f$ if
\[
f^\circ(\bx;\bd)
:=
\limsup_{\by\to \bx,\ t\downarrow 0}
\frac{f(\by+t\bd)-f(\by)}{t}
\ge 0
\quad \forall \bd \in \R^n,
\]
or equivalently, $0 \in \partial f(\bx)$.

Let $D_f \subset \R^n$ denote the set of differentiability points of $f$. Since $f$ is locally Lipschitz, $D_f$ has full measure. Throughout the paper, we denote by \(\Nz_0:=\{0,1,2,\ldots\}\) the set of nonnegative integers.

\paragraph{$\varepsilon$-Clarke stationary point.} For a given tolerance $\varepsilon>0$, if
\begin{equation}\label{eq:clarke}
\min_{\bv \in \partial f(\bx)} \|\bv\| \le \varepsilon,
\end{equation}
then we say that $\bx \in \R^n$ is an
\bfi{$\varepsilon$-Clarke stationary point}, where the Clarke subdifferential is given by
\begin{equation}\label{eq:clarkesub}
\partial f(\bx)
:=
\conv\Big\{
\bg \in \R^n \ \big|\ \exists \{\bx_k\}_{k\in\Nz_0} \subset D_f,\ 
\bx_k \to \bx,\ 
\nabla f(\bx_k) \to \bg
\Big\}.
\end{equation}

\paragraph{Computability issue.}
Condition \eqref{eq:clarke} is generally impractical to verify in derivative-free settings, since it requires full knowledge of $\partial f(\bx)$. Therefore, we rely on computable approximations of the Goldstein subdifferential using discrete gradients, as will be described in Section~\ref{subsec:goldstein_stationarity}.

\subsection{Goldstein Subdifferential and Approximate Stationarity}
\label{subsec:goldstein_stationarity}

In practice, verifying \eqref{eq:clarke} is challenging, since it
requires access to the entire set $\partial f(\bx)$, which is generally
unavailable in derivative-free settings. For this reason, it is more
practical to work with the \bfi{Goldstein $\varepsilon$-subdifferential}.
Here and throughout, $\cl A$ denotes the closure of a set $A$. We define
\begin{equation}\label{e.Goldsteinsubb}
\partial_\varepsilon f(\bx)
:=
\cl\conv\Big(
\bigcup_{\|\by-\bx\|\le \varepsilon} \partial f(\by)
\Big),
\end{equation}
where $\partial f(\by)$ denotes the Clarke subdifferential of $f$ at $\by$, defined by \eqref{eq:clarkesub}.  Unlike $\partial f(\bx)$, the set $\partial_\varepsilon f(\bx)$ captures
information from a neighborhood of $\bx$, and is therefore more amenable
to approximation via sampling and finite-difference constructions. For locally Lipschitz functions, $\partial_\varepsilon f(\bx)$ is a nonempty, convex, and compact subset of $\R^n$. 
Indeed, since $f$ is locally Lipschitz, the Clarke subdifferential $\partial f(\by)$ is nonempty and bounded on bounded sets. Therefore, the union $\bigcup_{\|\by-\bx\|\le \varepsilon} \partial f(\by)$ is bounded. In finite-dimensional spaces, the closed convex hull of a bounded set is compact, which implies that $\partial_\varepsilon f(\bx)$ is compact.

Recall that $\bx \in \mathbb{R}^n$ is a Clarke stationary point if $0 \in \partial f(\bx)$. Moreover, if $\bx$ is a local minimizer of $f$, then it is necessarily a Clarke stationary point. Since the function $f$ is locally Lipschitz but possibly nonsmooth,
the exact Clarke subdifferential $\partial f(\bx)$ is generally unavailable.
Therefore, we construct computable approximations of the Clarke and
Goldstein $\varepsilon$-subdifferentials using discrete-gradient constructions,
which are accurate up to a controlled approximation error.

\paragraph{Approximate stationarity point.}
Given $\delta>0$ and $\varepsilon>0$, let
$\conv\mathcal{G}_\varepsilon(\bx)$ be a finite computable model of
$\partial_\varepsilon f(\bx)$. We say that $\bx\in\R^n$ is a
$(\delta,\mathcal{G}_\varepsilon(\bx))$ \bfi{Goldstein stationary point} if
\[
\min\{\|\bv\|\mid \bv\in\conv\mathcal{G}_\varepsilon(\bx)\}\le \delta.
\]
This definition is understood as a computable surrogate of Goldstein
stationarity based on the finite model
$\conv\mathcal{G}_\varepsilon(\bx)$; cf. \cite{Maleknia2024}.

\paragraph{Relation to discrete-gradient approximations.}
The set $\mathcal{G}_\varepsilon(\bx)$ represents an ideal finite model of
the Goldstein subdifferential $\partial_\varepsilon f(\bx)$ constructed
from exact subgradient information. In derivative-free settings, such
subgradients are generally unavailable. We therefore replace this ideal
model by a finite set $\mathcal V(\bx_k)$ constructed via discrete
gradients.

The resulting set $\conv(\mathcal V(\bx_k))$ should not be interpreted as
an exact outer approximation of $\partial_{\varepsilon_k}f(\bx_k)$.
Rather, under the discrete-gradient consistency assumptions stated below,
its elements are close to elements of the Goldstein subdifferential:
\begin{equation}\label{eq:outer_dg}
\conv(\mathcal V(\bx_k))
\subset
\partial_{\varepsilon_k} f(\bx_k)+B(0,\delta_k),
\end{equation}
where $\delta_k\to0$ as the discretization parameter vanishes. Thus,
$\conv(\mathcal V(\bx_k))$ is a computable sampled model of
$\partial_{\varepsilon_k}f(\bx_k)$, accurate up to the approximation
error $\delta_k$.

\paragraph{Approximate stationarity via discrete gradients.}
Let $\mathcal V(\bx_k)$ be the finite set of discrete gradients constructed
at $\bx_k$, and assume that \eqref{eq:outer_dg} holds. We say that
$\bx_k \in \R^n$ is a
$(\delta,\mathcal V(\bx_k))$ \bfi{discrete-gradient Goldstein stationary point} if
\begin{equation}\label{eq:goldstein_stationarity_dg}
\min\{\|\bv\|\mid \bv\in \conv(\mathcal V(\bx_k))\}\le \delta.
\end{equation}
This is a computable stationarity test based on the discrete-gradient
model. Its relation to true Goldstein stationarity depends on the
approximation error in \eqref{eq:outer_dg}.

\begin{prop}[Consistency with Goldstein stationarity]
\label{prop:DG_to_Goldstein}
Let $f:\R^n \to \R$ be locally Lipschitz and fix $\varepsilon>0$.
Assume that the finite set $\mathcal V(\bx_k)$ of discrete gradients satisfies
\eqref{eq:outer_dg} for some $\delta_k \ge 0$. Let
$\mathcal G_\varepsilon(\bx_k)$ be a nonempty finite model of
$\partial_\varepsilon f(\bx_k)$ satisfying the approximate covering condition
\begin{equation}\label{eq:approx_g}
\partial_\varepsilon f(\bx_k)
\subset
\conv(\mathcal G_\varepsilon(\bx_k)) + B(0,\eta_k),
\end{equation}
for some $\eta_k \ge 0$. If $\bx_k$ is a
$(\delta,\mathcal V(\bx_k))$ discrete-gradient Goldstein stationary point, i.e.,
\eqref{eq:goldstein_stationarity_dg} holds, then $\bx_k$ is a
$(\delta+\delta_k+\eta_k,\mathcal G_\varepsilon(\bx_k))$ Goldstein
stationary point, namely,
\[
\min_{\bu \in \conv(\mathcal G_\varepsilon(\bx_k))} \|\bu\|
\le
\delta + \delta_k + \eta_k.
\]
\end{prop}

\begin{proof}
By \eqref{eq:outer_dg}, \eqref{eq:goldstein_stationarity_dg}, and \eqref{eq:approx_g}, there are $\bv_k \in \conv(\mathcal V(\bx_k))$ with $\|\bv_k\| \le \delta$, $\bw_k \in \partial_\varepsilon f(\bx_k)$ with $\|\bv_k-\bw_k\| \le \delta_k$, and 
$\bu_k \in \conv(\mathcal G_\varepsilon(\bx_k))$ with $\|\bw_k-\bu_k\| \le \eta_k$, respectively. Therefore, by the triangle inequality,
\[
\|\bu_k\|
\le
\|\bv_k\|+\|\bv_k-\bw_k\|+\|\bw_k-\bu_k\|
\le
\delta+\delta_k+\eta_k.
\]
Since $\bu_k \in \conv(\mathcal G_\varepsilon(\bx_k))$, it follows that $\D\min_{\bu \in \conv(\mathcal G_\varepsilon(\bx_k))}\|\bu\|
\le
\|\bu_k\|
\le
\delta+\delta_k+\eta_k$. This proves the claim.
\end{proof}

\begin{rem}
The discrete-gradient approximation property used in
\eqref{eq:outer_dg} and in Proposition~\ref{prop:Vk_consistency} is
essential. It is not implied by local Lipschitz continuity alone. In the
discrete-gradient literature, such approximation results are obtained
under additional regularity conditions \cite{Bagirov2007,Karmitsa2012}, for example semismoothness together
with a uniform small-o condition on the discrete-gradient construction.
\end{rem}
\begin{rem}[Asymptotic consistency]
\label{rem:asymptotic_consistency}
Assume that $\conv(\mathcal V(\bx_k))
\subset
\partial_\varepsilon f(\bx_k)+B(0,\rho_k)$ as $\rho_k\to0$, and $\partial_\varepsilon f(\bx_k)
\subset
\conv(\mathcal G_\varepsilon(\bx_k))+B(0,\eta_k)$ as $\eta_k\to0$. If $\bx_k$ is a sequence of
$(\tau_k,\mathcal V(\bx_k))$ discrete-gradient Goldstein stationary points,
with $\tau_k\to0$, then
\[
\min_{\bu\in\conv(\mathcal G_\varepsilon(\bx_k))}\|\bu\|
\le
\tau_k+\rho_k+\eta_k
\to0.
\]
Moreover, $\dist\bigl(0,\partial_\varepsilon f(\bx_k)\bigr)
\le
\tau_k+\rho_k
\to0$.
\end{rem}

\subsection{Discrete-gradient Construction}
\label{subsec:discrete_gradient}

In this subsection, we describe the construction of the discrete gradient following the approach introduced in \cite{Bagirov2007,Karmitsa2012}. We denote the set of unit directions by
\[
D := \{ \bd=(d_1,\dots,d_n) \in \R^n \mid \|\bd\| = 1 \}.
\]
For a given direction $\bd \in D$, we define the index
\begin{equation}\label{e.imax}
i_{\max} := \argmax_{j=1,\dots,n} |d_j|.
\end{equation}
Since $\|\bd\|=1$, it follows that $|d_{i_{\max}}| \ge 1/\sqrt{n}$ and hence $d_{i_{\max}} \neq 0$, which guarantees that the construction below is well-defined. The set of vertices of the unit hypercube
\[
G := \{ \ve \in \R^n \mid \ve = (e_1,\dots,e_n)^\top,\ |e_j| = 1,\ j=1,2,...,n \},
\]
and the class of positive infinitesimal functions
\[
P := \left\{ z:(0,\infty)\to(0,\infty) \ \middle|\ \frac{z(\lambda)}{\lambda} \to 0 \text{ as } \lambda \to 0^+ \right\}
\]
are defined. Given parameters $\lambda>0$, $\omega\in(0,1]$, $\ve\in G$, and $z\in P$, the sequence of points
\begin{equation}\label{eq:xgrad_alg}
\bx_0 := \bx + \lambda \bd,
\qquad
\bx_j := \bx_0 + z(\lambda) \ve_j(\omega),
\qquad j=1,\dots,n,
\end{equation}
is defined, where
\[
\ve_1(\omega) :=
\begin{pmatrix}
\omega e_1 \\ 0 \\ \vdots \\ 0
\end{pmatrix},
\quad
\ve_2(\omega):=
\begin{pmatrix}
\omega e_1 \\ \omega^2 e_2 \\ 0 \\ \vdots \\ 0
\end{pmatrix},
\quad
\ldots,
\quad
\ve_n(\omega):=
\begin{pmatrix}
\omega e_1 \\ \omega^2 e_2 \\ \vdots \\ \omega^n e_n
\end{pmatrix}.
\]

In practical implementations, a specific choice of $z(\lambda)$ must be made. Following the implementation commonly used in \cite{Bagirov2007,Karmitsa2012}, we select the quadratic perturbation
\begin{equation}\label{eq:zlambda}
z(\lambda) = \lambda^2,
\end{equation}
which satisfies $z(\lambda)/\lambda \to 0$ as $\lambda\to0$ and yields a stable numerical approximation.

The discrete gradient associated with index $i_{\max}$ satisfying \eqref{e.imax} is defined as
\[
\bg_{i_{\max}}(\bx,\bd,\ve,z,\lambda,\omega)
:=
(g_{i_{\max}1},\dots,g_{i_{\max}n})^\top,
\]
with components
\[
g_{i_{\max}j} :=
\frac{f(\bx_j)-f(\bx_{j-1})}
     {z(\lambda)\omega^j e_j},
\qquad j\neq i_{\max},
\]
and
\[
g_{i_{\max}i_{\max}} :=
\frac{
f(\bx+\lambda \bd) - f(\bx)
- \lambda \displaystyle\sum_{j\neq i_{\max}} g_{i_{\max}j} d_j
}
{\lambda d_{i_{\max}}}.
\]

Since $z(\lambda)>0$, $\omega>0$, and $e_j \in \{\pm1\}$, $j=1,2,...,n$, all denominators are nonzero and the above expressions are well-defined. By construction, the identity
\[
f(\bx+\lambda \bd) - f(\bx)
=
\lambda \bd^\top \bg_{i_{\max}}(\bx,\bd,\ve,z,\lambda,\omega)
\]
holds exactly. Each discrete gradient requires at most $n+2$ function evaluations (or $n+1$ additional evaluations if $f(\bx)$ is already available).

\subsection{Discrete-Gradient Sets and Finite Subdifferential Approximation}
\label{subsec:dg_sets}

\paragraph{Discrete gradient set.}

For a fixed scalar $\lambda>0$, we define
\[
V(\bx,\lambda)
:=
\cl\conv
\left\{
\bg_{i_{\max}}(\bx,\bd,\ve,z,\lambda,\omega)
\mid
\bd\in D,\, \ve\in G,\, z\in P,\, \omega\in(0,1]
\right\}.
\]
The set $V(\bx,\lambda)$ coincides with the discrete-gradient
set $D_0(\bx,\lambda)$ introduced in \cite{Bagirov2007},
up to notation. By Proposition~5.1 in \cite{Bagirov2007}, if $f$ is Lipschitz continuous with constant $L$, then
\begin{equation}\label{eq:dg_bounded}
\|\bg_{i_{\max}}(\bx,\bd,\ve,z,\lambda,\omega)\|
\le C(n)L,
\end{equation}
with the dimension-dependent constant $C(n)>0$, and consequently,
\begin{equation}\label{eq:V_bounded}
V(\bx,\lambda) \subset B(0,C(n)L).
\end{equation}

\paragraph{Approximation of the Clarke subdifferential.}
By Corollary~5.1 in \cite{Bagirov2007}, under the additional regularity
conditions used there, for each fixed $\bx$ and every sufficiently small scalar
$\lambda>0$,
\begin{equation}\label{eq:dg_clarke}
V(\bx,\lambda)
\subset
\partial f(\bx) + B(0,\delta(\lambda)),
\qquad \delta(\lambda)\to 0 \text{ as } \lambda\to 0.
\end{equation}
This is a pointwise approximation result. For convergence arguments
involving points in a neighborhood of $\bx$, we impose Assumption~(A2),
which is the neighborhood version used in \cite[Assumption~5.1]{Bagirov2007}. Moreover, if $\bx^*$ is a local minimizer, then by Proposition~5.3 in \cite{Bagirov2007}, there exists $\lambda_0>0$ such that
\begin{equation}\label{eq:dg_optimality}
0 \in V(\bx^*,\lambda)
\qquad \text{for all } \lambda \in (0,\lambda_0).
\end{equation}

\paragraph{Approximation of the Goldstein subdifferential.}

For $\varepsilon>0$, we define
\[
V_\varepsilon(\bx,\lambda)
:=
\conv
\left(
\bigcup_{\|\by-\bx\|\le \varepsilon} 
V(\by,\lambda)
\right).
\]
If the pointwise approximation \eqref{eq:dg_clarke} holds uniformly for
all $\by$ satisfying $\|\by-\bx\|\le\varepsilon$, then combining it with
the definition of $\partial_\varepsilon f(\bx)$ yields
\begin{equation}\label{eq:dg_goldstein}
V_\varepsilon(\bx,\lambda)
\subset
\partial_\varepsilon f(\bx) + B(0,\delta(\lambda)).
\end{equation}
Under the corresponding uniform version of the pointwise approximation
\eqref{eq:dg_clarke} over the ball
$\{\by\mid \|\by-\bx\|\le \varepsilon\}$, the ideal neighborhood-based
discrete-gradient set satisfies
\begin{equation}\label{eq:dg_ideal}
V_\varepsilon(\bx,\lambda)
\subset
\partial_\varepsilon f(\bx)+B(0,\delta(\lambda)),
\qquad
\delta(\lambda)\to0 .
\end{equation}
Thus, as $\lambda\to0$, the ideal discrete-gradient model is
asymptotically consistent with the Goldstein subdifferential.

\paragraph{Standing Assumptions.}
We collect here all assumptions required for the convergence analysis.

\begin{itemize}
\item[(A1)] The function $f:\Rz^n \to \R$ is locally Lipschitz.

\item[(A2)] For every $\bx\in \Rz^n$ and every $\delta>0$, there exist
$\lambda_0>0$ and $\eps>0$ such that
\[
V(\by,\lambda)
\subseteq
\partial f(\bx+\ol B(0,\delta))+B(0,\delta)
\qquad
\forall \by\in B(\bx,\eps),\quad
\forall \lambda\in(0,\lambda_0),
\]
where
\[
\partial f(\bx+\ol B(0,\delta))
:=
\bigcup_{\bz\in\ol B_\delta(\bx)}\partial f(\bz),
\qquad
\ol B(\bx,\delta)
:=
\{\bz\in \Rz^n\mid \|\bz-\bx\|\le \delta\}.
\]

\item[(A3)] The initial level set
\[
\mathcal L(\bx_0)
:=
\{\bx\in\Rz^n \mid f(\bx)\le f(\bx_0)\}
\]
is bounded.

\item[(A4)] The function $f$ is weakly upper semismooth in the relevant
search directions. Moreover, the discrete gradients generated along any \texttt{LiS} direction satisfy the following directional consistency
condition: for any direction $\bd\in\Rz^n\setminus\{0\}$, any
$z\in\Rz^n$, any infinite index set $J\subset\mathbb N$, any sequence
$\{h_i\}_{i\in J}\subset\R_+$ with $h_i\downarrow0$ as $i\to\infty$
along $J$, and any sequence of discrete gradients $\{\bg_i\}_{i\in J}$
generated at the points $z+h_i\bd$, one has
\begin{equation}\label{ass:wus_dg_lis}
\limsup_{i\in J,\ i\to\infty}\bg_i^\top\bd
\ge
\liminf_{i\in J,\ i\to\infty}
\frac{f(z+h_i\bd)-f(z)}{h_i}.
\end{equation}

\item[(A5)] At every iteration where the local Armijo admissibility result
is invoked, the current model vector
$\bv_k\in\Rz^n\setminus\{0\}$ and the corrected direction
$\bd_k\in\Rz^n\setminus\{0\}$ satisfy the directional model-error bound
\begin{equation}\label{eq:dir_error_bound}
\left|
f^\circ(\bx_k;\bd_k)
-
\bv_k^\top\bd_k
\right|
\le
\xi_k\|\bd_k\|,
\end{equation}
where $\xi_k\ge0$. Moreover, there exists $\theta\in(0,1)$ such that the relative directional model-error condition
\begin{equation}\label{eq:dir_error_small}
\xi_k
\le
\theta\varrho\|\bv_k\|
\end{equation}
holds. This assumption is used only to guarantee local Armijo admissibility of
the corrected direction. It is not required for the finite termination of
Algorithm~\ref{AlgTPLSDG}, which is proved under (A4).

\end{itemize}
Assumption (A2) is the discrete-gradient approximation assumption used in
\cite[Assumption~5.1]{Bagirov2007}. It is not a consequence of local
Lipschitz continuity alone; it is imposed to ensure that discrete gradients
computed at nearby points and sufficiently small discretization parameters
are contained in a small enlargement of the Clarke subdifferential over a
neighborhood of $\bx$.

\begin{rem}[Clarke subdifferential properties used below]
\label{rem:clarke_properties}
Under (A1), the Clarke subdifferential $\partial f(\bx)$ is nonempty,
convex, and compact for every $\bx\in\Rz^n$. Moreover, the set-valued
mapping $\bx\mapsto\partial f(\bx)$ is outer semicontinuous and locally
bounded. In particular, if $\bx_j\to\bx$, $\bg_j\in\partial f(\bx_j)$, and
$\bg_j\to\bg$, then $\bg\in\partial f(\bx)$. Also, for every bounded set $C\subset\Rz^n$, there exists $M>0$ such that
\[
\|\bg\|\le M,
\qquad
\forall \bg\in\partial f(\bx),\quad \forall \bx\in C.
\]
These standard facts will be used in the Clarke-stationarity convergence
argument; see, e.g., \cite{Clarke1990}.
\end{rem}
\paragraph{Finite approximation via sampling.}

In practice, the full sets $V(\bx,\lambda)$ and $V_\varepsilon(\bx,\lambda)$ are unavailable. We construct a finite approximation by selecting directions $\{\bd_1,\dots,\bd_m\}\subset D$ and sampling points $\by_\ell$ such that $\|\by_\ell-\bx_k\|\le \varepsilon_k$:
\begin{equation}\label{eq:defmatVk}
\mathcal{V}(\bx_k)
:=
\left\{
\bg_{i_{\max}}\left(\by_\ell,\bd_\ell,\ve_\ell,z_k,\lambda_k,\omega\right)
\right\}_{\ell=1}^m.
\end{equation}
Then,
\begin{equation}\label{eq:Vk_inclusion}
\conv(\mathcal{V}(\bx_k))
\subset
V_{\varepsilon_k}(\bx_k,\lambda_k).
\end{equation}
Indeed, $\conv(\mathcal V(\bx_k))$ provides a computable finite sampled
model whose elements are controlled through the discrete-gradient
approximation property. It should not be interpreted as an outer
approximation of $\partial_{\varepsilon_k} f(\bx_k)$ in the set-inclusion
sense. Combining \eqref{eq:dg_goldstein} and \eqref{eq:Vk_inclusion}
yields the following finite-sample consistency result.

\begin{prop}[Finite-sample consistency of discrete gradients] \label{prop:Vk_consistency}
Assume that the pointwise discrete-gradient approximation property
\eqref{eq:dg_clarke} holds at each sampling point $\by_\ell$.
Let $\{\lambda_k\}_{k\in\Nz_0}$ and $\{\varepsilon_k\}_{k\in\Nz_0}$ satisfy $\lambda_k \to 0$ and $\varepsilon_k \to 0$. Let $\mathcal{V}(\bx_k)$ be a finite set of discrete gradients constructed
at sampling points $\{\by_\ell\}_{\ell=1}^{m_k}$ satisfying
\[
\|\by_\ell-\bx_k\|\le \varepsilon_k,
\qquad \ell=1,\dots,m_k,
\]
using discretization parameter $\lambda_k$. Suppose that the corresponding
pointwise errors are bounded by $\delta_k\to0$, i.e.,
\[
V(\by_\ell,\lambda_k)
\subset
\partial f(\by_\ell)+B(0,\delta_k),
\qquad \ell=1,\dots,m_k.
\]
Then
\begin{equation}\label{eq:Vk_goldstein}
\conv(\mathcal{V}(\bx_k))
\subset
\partial_{\varepsilon_k} f(\bx_k) + B(0,\delta_k).
\end{equation}
\end{prop}
\begin{proof}
Each element $\bg_\ell \in \mathcal{V}(\bx_k)$ is a discrete gradient computed at a point $\by_\ell$ with $\|\by_\ell - \bx_k\| \le \varepsilon_k$, and  a parameter $\lambda_k$. By the discrete-gradient approximation property \eqref{eq:dg_clarke}, for each $\ell$ there exists $\bw_\ell \in \partial f(\by_\ell)$ such that $\|\bg_\ell - \bw_\ell\|
\le \delta_k$. Hence,
\[
\bg_\ell \in \partial f(\by_\ell) + B(0,\delta_k).
\]
Taking convex combinations, for any $\bv \in \conv(\mathcal{V}(\bx_k))$, we can write
\[
\bv = \sum_{\ell=1}^{m_k} p_\ell \bg_\ell,
\qquad
\sum_{\ell=1}^{m_k} p_\ell = 1,\qquad p_\ell \ge 0.
\]
For each $\ell$, there exists $\bw_\ell \in \partial f(\by_\ell)$ such that
$\|\bg_\ell - \bw_\ell\| \le \delta_k$, hence $\bg_\ell = \bw_\ell + \be_\ell$
with $\be_\ell \in B(0,\delta_k)$. Therefore,
\[
\bv
= \sum_{\ell=1}^{m_k} p_\ell \bw_\ell
+ \sum_{\ell=1}^{m_k} p_\ell \be_\ell.
\]
Since $\bw_\ell \in \partial f(\by_\ell)$ for each $\ell$, it follows that
\[
\sum_{\ell=1}^{m_k} p_\ell \bw_\ell 
\in 
\conv\Big(\bigcup_{\ell=1}^{m_k} \partial f(\by_\ell)\Big),
\]
and since $B(0,\delta_k)$ is convex, $\D\sum_{\ell=1}^{m_k} p_\ell \be_\ell \in B(0,\delta_k)$. Hence,
\[
\bv \in
\conv\Big( \bigcup_{\ell=1}^{m_k} \partial f(\by_\ell) \Big)
+ B(0,\delta_k),
\]
which implies $\conv(\mathcal{V}(\bx_k))
\subset
\conv\Big( \bigcup_{\ell=1}^{m_k} \partial f(\by_\ell) \Big)
+ B(0,\delta_k)$. Since $\|\by_\ell - \bx_k\| \le \varepsilon_k$, it follows that
\[
\bigcup_{\ell=1}^{m_k} \partial f(\by_\ell)
\subset
\bigcup_{\|\by-\bx_k\|\le \varepsilon_k} \partial f(\by).
\]
Therefore,
\[
\conv(\mathcal{V}(\bx_k))
\subset
\cl\conv\Big(
\bigcup_{\|\by-\bx_k\|\le \varepsilon_k} \partial f(\by)
\Big)
+ B(0,\delta_k)
\overset{\eqref{e.Goldsteinsubb}}{=}
\partial_{\varepsilon_k} f(\bx_k)
+ B(0,\delta_k),
\]
so that \eqref{eq:Vk_goldstein} is obtained.
\end{proof}

\begin{rem}[Role of discrete-gradient properties]
\label{rem:dg_roles}
The inclusion \eqref{eq:dg_clarke} is the key property used in the proof above. It ensures that discrete gradients are contained in a controlled neighborhood of the Clarke subdifferential.

The optimality condition \eqref{eq:dg_optimality}, which guarantees that $0 \in V(\bx^*,\lambda)$ at local minimizers, is not required for the present approximation result, but will play a role later in the stationarity and convergence analysis.
\end{rem}

\paragraph{Discrete-gradient model error and sources of approximation.}
Assumption~(A2) provides the basic consistency property of the
discrete-gradient construction. In particular, discrete gradients computed at
points sufficiently close to a reference point, and with sufficiently small
discretization parameter, belong to a small enlargement of the Clarke
subdifferential over a neighborhood of that reference point. Therefore, the
error appearing in the discrete-gradient approximation is not noise in the
objective values. Rather, it is a deterministic model error caused by replacing
exact Clarke subgradient information by computable discrete gradients. In the
present framework, this approximation is governed by three quantities: the
Goldstein radius $\varepsilon_k$, the discretization error $\delta_k^{\mathrm{disc}} := \delta(\lambda_k)$, arising from the discrete-gradient construction, and the sampling/modeling
error $\delta_k^{\mathrm{samp}}$, accounting for possible finite-sampling,
inexact construction, or model-reduction effects in the finite set
$\mathcal V(\bx_k)$. Accordingly, the total approximation error is $\bar{\delta}_k
:=
\delta_k^{\mathrm{disc}}
+
\delta_k^{\mathrm{samp}}$.

The discretization error controls the ideal neighborhood-based discrete-gradient
set in the sense that
\[
V_{\varepsilon_k}(\bx_k,\lambda_k)
\subset
\partial_{\varepsilon_k} f(\bx_k)
+
B(0,\delta_k^{\mathrm{disc}}),
\]
provided that the pointwise discrete-gradient approximation error is uniform over
the ball $\{\by:\|\by-\bx_k\|\le\varepsilon_k\}$. The finite sampled model is
assumed to satisfy
\[
\conv(\mathcal V(\bx_k))
\subset
V_{\varepsilon_k}(\bx_k,\lambda_k)
+
B(0,\delta_k^{\mathrm{samp}}).
\]
In the exact finite-sampling case considered in \eqref{eq:Vk_inclusion}, this
relation holds with $\delta_k^{\mathrm{samp}}=0$. A positive
$\delta_k^{\mathrm{samp}}$ allows for additional finite-sampling, inexact
construction, or model-reduction errors. Combining the two relations yields
the finite-sample approximation property
\begin{equation}\label{eq:Vk_goldstein_bar}
\conv(\mathcal{V}(\bx_k))
\subset
\partial_{\varepsilon_k} f(\bx_k) + B(0,\bar{\delta}_k).
\end{equation}
Thus, $\bar{\delta}_k$ should be interpreted as the total model approximation
error between the finite computable discrete-gradient model
$\conv(\mathcal V(\bx_k))$ and the Goldstein subdifferential model
$\partial_{\varepsilon_k}f(\bx_k)$.

The overall approximation quality is governed by
$\varepsilon_k$, $\delta_k^{\mathrm{disc}}$, and
$\delta_k^{\mathrm{samp}}$. We assume that
\[
\delta_k^{\mathrm{disc}}\to0,
\quad \text{and} \quad
\delta_k^{\mathrm{samp}}\to0,
\]
as $\lambda_k\to0$ and the sampling/model construction becomes sufficiently
accurate. In statements involving the finite sampled set $\mathcal V(\bx_k)$,
we use the total error $\bar{\delta}_k$, while results concerning the ideal
neighborhood-based discrete-gradient construction
$V_{\varepsilon_k}(\bx_k,\lambda_k)$ involve only the discretization error
$\delta_k^{\mathrm{disc}}$.

Consequently, if $\min_{\bv\in\conv(\mathcal V(\bx_k))}\|\bv\|\to0$, $\varepsilon_k\to0$, $\bar\delta_k\to0$, then $\dist(0,\partial_{\varepsilon_k}f(\bx_k))\to0$. Indeed, by \eqref{eq:Vk_goldstein_bar}, for every
$\bv_k\in\conv(\mathcal V(\bx_k))$ there exists
$\bw_k\in\partial_{\varepsilon_k}f(\bx_k)$ such that
\[
\|\bv_k-\bw_k\|\le \bar\delta_k.
\]
Hence,
\begin{equation}\label{eq:dg_stationarity_residual}
\dist(0,\partial_{\varepsilon_k}f(\bx_k))
\le
\|\bv_k\|+\bar\delta_k.
\end{equation}
Taking $\bv_k$ as a minimal-norm element of
$\conv(\mathcal V(\bx_k))$ gives the assertion. Therefore, if there exists a subsequence $K\subset\mathbb N$ such that
\[
\bx_k\to\bx^\ast,
\quad
\min_{\bv\in\conv(\mathcal V(\bx_k))}\|\bv\|\to0,
\quad
\varepsilon_k\to0,
\quad \text{and} \quad
\bar\delta_k\to0
\quad
(k\in K),
\]
then the standard outer-semicontinuity argument for the Clarke
subdifferential yields $0\in\partial f(\bx^\ast)$, so that $\bx^\ast$ is a Clarke stationary. If, instead, only a uniform bound $\bar\delta_k\le \bar\delta$ is available, then exact Clarke stationarity cannot be concluded in general.
The conclusion is only approximate Goldstein--Clarke stationarity, with a
stationarity residual controlled by $\bar\delta$.

\subsection{Minimal-Norm Approximation and Quadratic Programming Formulation}
\label{subsec:min_norm_qp}

\paragraph{Approximate minimal-norm problem.}
In view of the inclusion \eqref{eq:Vk_goldstein}, we consider the problem of computing the minimal-norm element
of $\conv(\mathcal{V}(\bx_k))$ as a computable surrogate of the
minimal-norm element of $\partial_{\varepsilon_k} f(\bx_k)$.

\paragraph{Minimal-norm element.} 
Let $\mathcal{V}(\bx_k) = \{\bg_1,\bg_2,\dots,\bg_m\}$ and define the matrix
\[
V_k := \big[ \bg_1 \ \bg_2\ \cdots \ \bg_m \big] \in \R^{n\times m}.
\]
Any $\bv \in \conv(\mathcal{V}(\bx_k))$ can be written as $\bv = V_k \bp$, where $\bp \in \R^m$ satisfies
\[
\sum_{\ell=1}^m p_\ell = 1,\qquad p_\ell\geq0.
\]

The vector $\bv_k$ is the projection of the origin onto $\conv(\mathcal V(\bx_k))$:
\begin{equation}\label{eq:min_norm_dg_projection}
\bv_k = \operatorname{proj}_{\conv(\mathcal V(\bx_k))}(0),
\end{equation}
and $\|\bv_k\|$ represents the distance from the origin to the sampled approximation of the subdifferential. Since $\conv(\mathcal V(\bx_k))$ is a nonempty compact convex set, the
projection of the origin \eqref{eq:min_norm_dg_projection} onto this set exists and is unique. Hence, the problem
\begin{equation}\label{eq:prob_vk_goldstein}
\bv_k
=
\argmin_{\bv \in \conv(\mathcal{V}(\bx_k))}
\|\bv\|
\end{equation}
admits a unique solution $\bv_k=V_k\bp^\ast$. However, the coefficient vector  $\bp^\ast$ solving the simplex-constrained quadratic program need not be unique when the columns of \(V_k\) are linearly dependent.

Finding the approximate minimal-norm element \eqref{eq:prob_vk_goldstein} is equivalent to the convex quadratic program
\begin{equation}\label{QPalpha}
\min_{\bp \in \R^m}
\frac12\|V_k\bp\|^2
\quad
\text{s.t.}
\sum_{\ell=1}^m p_\ell = 1,\qquad p_\ell\geq0.
\end{equation}
Using $\|V_k \bp\|_2^2 = \bp^\top V_k^\top V_k \bp$, we define the Gram matrix
\[
G_k := V_k^\top V_k \in \R^{m\times m}.
\]
Equivalently, with this $G_k$, the problem \eqref{QPalpha} can be written as
\[
\min_{\bp\in\mathbb R^m}
\frac12 \bp^\top G_k\bp
\quad
\text{s.t.}
\quad
\sum_{\ell=1}^m p_\ell=1,\quad p_\ell\ge0,
\]
which is a convex quadratic program on the probability simplex. Let $\bp^\ast$ denote an optimal solution of \eqref{QPalpha}. The approximate minimal-norm element is $\bv_k = V_k \bp^\ast$. By \eqref{eq:Vk_goldstein}, $\bv_k$ can be interpreted as an approximate minimal-norm element of $\partial_{\varepsilon_k} f(\bx_k)$, up to an error of order $\delta_k$.

\begin{comment}
\paragraph{Remark on conditioning.}
The numerical stability of \eqref{QPalpha} depends on the conditioning of $G_k = V_k^\top V_k$. Approximate linear dependence among $\{\bg_i\}$ or large magnitude imbalances may lead to ill-conditioning, affecting the reliability of $\bp^\ast$. Practical stabilization strategies will be discussed in Section~\ref{MainCSG}.
\end{comment}

\subsection{Approximate Descent Directions and Existing Approaches}
\label{subsec:approx_descent}

\paragraph{Three-layer structure.}

The proposed framework can be decomposed into three conceptual layers:
\begin{itemize}
\item[(i)] \bfi{Sampling layer:} discrete gradients are computed along directions $\bd \in D$ with $\|\bd\|=1$ to construct a finite set $\mathcal{V}(\bx_k)$;

\item[(ii)] \bfi{Model layer:} the convex hull $\conv(\mathcal{V}(\bx_k))$ is used as a computable finite sampled model associated with the Goldstein subdifferential $\partial_{\varepsilon_k} f(\bx_k)$;

\item[(iii)] \bfi{Optimization layer:} a minimal-norm stationarity vector $\bv_k$ is computed by solving \eqref{eq:prob_vk_goldstein}.
\end{itemize}

The normalization constraint $\|\bd\|=1$ is required only in the sampling layer and does not restrict the optimization layer.

\paragraph{Approximate descent directions.}

For a locally Lipschitz function $f$, the computation of a descent direction can be formulated via the minimal-norm element of the Goldstein $\varepsilon$-subdifferential:
\begin{equation}\label{eq:prob1}
    \widehat{\bv}
    :=
    \argmin_{\bv \in \partial_{\varepsilon} f(\bx)}
    \|\bv\|.
\end{equation}
If $\widehat{\bv} \neq 0$, the corresponding $\varepsilon$-steepest descent direction is $\bd_s := -\widehat{\bv}/\|\widehat{\bv}\|$. However, solving \eqref{eq:prob1} is generally impractical in derivative-free settings, since it requires full knowledge of $\partial_\varepsilon f(\bx)$. To overcome this limitation, discrete approximations can be employed. Let $V(\bx,\lambda)$ denote a discrete-gradient set. It is known (see \cite[Proposition~5.4]{Bagirov2007}) that if
\[
\bv^\lambda
:=
\argmin_{\bv \in V(\bx,\lambda)} \|\bv\|
\quad \text{and} \quad \bv^\lambda \neq 0,
\]
then the normalized direction $\bd^\lambda := -\bv^\lambda/\|\bv^\lambda\|$ is a descent direction at $\bx$. This guarantee relies on access to the full set $V(\bx,\lambda)$.

In practice, only a finite subset is available, and the descent property may be degraded due to approximation errors.

A related strategy based on subgradient sampling has been proposed by Maleknia and Soleimani--Damaneh \cite{Maleknia2024}. They construct an inner approximation $\conv(\mathcal G_\varepsilon(\bx)) \subset \partial_\varepsilon f(\bx)$, and compute
\begin{equation}\label{eq:approx_min_norm}
\widehat{\bu}
:=
\argmin_{\bu \in \conv(\mathcal G_\varepsilon(\bx))}
\|\bu\|.
\end{equation}
The direction $\overline{\bd}_s := -\widehat{\bu}/\|\widehat{\bu}\|$ approximates the $\varepsilon$-steepest descent direction and admits a guaranteed enrichment mechanism based on exact subgradients.

\paragraph{Discrete-gradient-based directions.}

In the present work, we adopt a similar minimal-norm framework, but replace exact subgradients by discrete-gradient approximations. Specifically, we construct a finite set $\mathcal{V}(\bx_k)$ using only function evaluations and compute $\bv_k$ by \eqref{eq:prob_vk_goldstein}.

A natural reference direction associated with the minimal-norm model is $-\bv_k$. Indeed, since
$\bv_k$ is computed by \eqref{eq:prob_vk_goldstein}, the vector $-\bv_k$ may be interpreted as a computable approximation of an $\varepsilon$-steepest descent direction induced by the sampled
subdifferential model. Due to the approximation property \eqref{eq:Vk_goldstein_bar}, that is
\[
\conv(\mathcal{V}(\bx_k))
\subset
\partial_{\varepsilon_k} f(\bx_k) + B(0,\bar{\delta}_k),
\]
the vector $\bv_k$ provides approximate first-order information about
the local behavior of $f$. In particular, when the approximation error
$\bar\delta_k$ is sufficiently small, one heuristically expects
descent-oriented behavior along directions sufficiently aligned with
$-\bv_k$.

However, in the proposed framework, the vector $-\bv_k$ is not imposed
as the actual search direction. Instead, $\bv_k$ primarily serves as a
computable stationarity and orientation measure derived from the
approximate subdifferential model. The search direction may then be
generated using additional algorithmic mechanisms that exploit
information accumulated across iterations, while preserving suitable
descent properties relative to $\bv_k$.

This separation between the approximation model and the direction
generation mechanism is fundamental in the derivative-free setting.
It allows the incorporation of memory-based or curvature-enhanced
search directions without modifying the underlying stationarity
framework. In particular, Section~\ref{MainCSG} introduces matrix
conjugate subgradient directions constructed from the vectors
$\{\bv_k\}_{k\in\Nz_0}$ together with previous search information.

\section{Matrix Conjugate Subgradient Method}\label{MainCSG}

In this section, we extend the matrix parametric approach
(described in the supplemental material of the present paper~\cite[Section 1]{suppMat}) to a nonsmooth setting by employing
discrete-gradient approximations of the Goldstein
$\varepsilon$-subdifferential \cite{Bagirov2007,Karmitsa2012,Maleknia2024}. The resulting framework
is embedded into a globalization mechanism designed to promote stability
and convergence toward Clarke stationary points while remaining entirely
Hessian-free \cite{audet2006mesh,FasLLR14,Garmanjani2016}.

In pursuit of flexibility, diversity, and generality in optimization algorithms, we here employ a matrix-based formulation of the \texttt{CG} parameter in our algorithmic framework, as opposed to the traditional scalar parameter settings. This extended approach enhances a conjugacy factor, called the degree of conjugacy, by integrating additional recent search directions into the method, while ensuring global convergence via suitable safeguard mechanisms.

As is well-known from classical literature, to solve the unconstrained optimization problem \eqref{UP}, \texttt{CG} iterations are recursively updated by
\begin{equation}\label{xk1}
\bx_0 \in \mathbb{R}^n,\qquad \bx_{k+1} := \bx_k + \bs_k,\qquad \bs_k := t_k \bd_k^{\CG},
\end{equation}
for all $k \in \mathbb N_0$, in which $t_k \in (0,+\infty)$ is the step size computed via a \texttt{LiS} technique along the (descent) direction $\bd_k^{\CG}$, successively generated by 
\begin{equation}\label{dk1}
 \bd_{0}^{\CG} := -\nabla f_0, \qquad \bd_{k+1}^{\CG} := -\nabla f_{k+1} + \beta_k^{\CG} \bd_k^{\CG}, \qquad \forall k \in \mathbb N_0,
\end{equation}
with $\beta_k \in \mathbb{R}$ referred to as the \texttt{CG} parameter.

To the best of our knowledge, the computational performance of \texttt{CG} methods is strongly influenced by the particular choice of the parameter $\beta_k$ \cite{AndreiComparison}. Existing analytical developments aimed at improving the theoretical properties or numerical performance of classical \texttt{CG} algorithms have largely been based on scalar parameterizations \cite{SBKRairo,HagerZhang}, or, at most, on vector-based extensions \cite{ZhangZhouLiOMS}, which may be regarded as ``one-dimensional'' improvement strategies. In this sense, the current literature reveals a notable lack of analytical motivation for multidimensional generalizations. As previously mentioned, such extensions can offer several advantages, including increased flexibility of \texttt{CG} algorithms, the incorporation of richer theoretical structures, and, ultimately, improved robustness of the resulting methods.

Here, we present a general formulation for parameterizing the \texttt{CG} direction \eqref{dk1} in a matrix-based framework. We begin by formally replacing the scalar parameter $\beta_k$ with $\beta_k \mathbf{I}$, where $\mathbf{I} \in \mathbb{R}^{n \times n}$ denotes the identity matrix. To further enhance the dimensionality of the parametric adjustment—and consequently increase the flexibility of the algorithm—we extend this formulation by introducing a matrix parameter $\boldsymbol{\ss}_k \in \mathbb{R}^{n \times n}$ in place of $\beta_k \mathbf{I}$. Accordingly, the search direction \eqref{dk1} is reformulated as
\begin{eqnarray}\label{dkB1}
 \bd_{0}^{\CG} := -\nabla f_{0}, \qquad \bd_{k+1}^{\CG} := -\nabla f_{k+1} + \boldsymbol{\ss}_k \bd_k^{\CG}, \qquad \forall k \in \mathbb N_0.
\end{eqnarray}

It is worth noting, however, that employing a dense matrix $\boldsymbol{\ss}_k$ in \eqref{dkB1} may incur significant computational overhead, particularly for large-scale problems. To maintain computational efficiency, it is therefore natural and practical to impose sparsity on $\boldsymbol{\ss}_k$, or alternatively, to adopt a limited-memory form of this matrix. 

To ensure the global convergence of the corresponding \texttt{CG} algorithm, it is necessary to impose appropriate safeguards on the search direction formula given by \eqref{dkB1}. These include preserving the SDC and ensuring that the sequence $\{\|\boldsymbol{\ss}_k \bd_k\|\}_{k \geq 0}$ remains uniformly bounded \cite{SugikiNarushimaYabe}. Such conditions can, for instance, be enforced through suitable restart schemes \cite{DaiLiaoNCG}. On the other hand, determining an appropriate sparsity structure for the matrix $\boldsymbol{\ss}_k$, as well as devising an effective strategy for computing its entries based on key properties of limited-memory algorithms, is of significant importance. In this regard, one possible approach is outlined in \cite[Section 1]{suppMat}, where $\boldsymbol{\ss}_k$ is first modeled as a banded tridiagonal matrix, and its entries are subsequently determined by exploiting more recent search directions to enhance conjugacy.

Now, we discuss how the proposed \texttt{CG} direction can be incorporated into our conjugate subgradient algorithm. To proceed, first by replacing $\nabla f_k$ in \eqref{dkB1} by $\bv_k$ defined by solving \eqref{eq:prob_vk_goldstein}, we define the \texttt{MatCSG} direction as
\begin{equation}\label{dkCSG}
\bd_0^{\MCS} := -\bv_0, \qquad 
\bd_{k+1}^{\MCS} := -\bv_{k+1} + \boldsymbol{\ss}_k \bd_k^{\MCS},
\end{equation}
where $\boldsymbol{\ss}_k\in\R^{n\times n}$ is a sparse matrix parameter as already introduced (see also Section 1 in \cite{suppMat}). Accordingly, we define the subgradient displacement vector as follows by 
\[
\bq_k := \bv_{k+1} - \bv_k.
\]

Unlike classical \texttt{CG} methods, where $\nabla f_k$ represents exact first-order information, the vectors $\bv_k$ here are approximate subgradient models constructed from sampled discrete gradients. Consequently, the term $\boldsymbol{\ss}_k \bd_k^{\MCS}$ should be interpreted as a heuristic memory mechanism that enriches the search direction using past information, rather than enforcing exact conjugacy.

Since $f$ is locally Lipschitz, the exact first-order directional behavior is characterized by the Clarke directional derivative
\[
f^\circ(\bx;\bd)
=
\max_{\bg\in\partial f(\bx)}
\bg^\top \bd.
\]
However, in the present framework, the vectors $\bv_k$ are computed from sampled discrete gradients and therefore only provide approximate first-order information. Consequently, the directional quantity
\[
\bv_{k+1}^\top \wt\bd_{k+1}^{\MCS}
\]
cannot, in general, be identified with the exact directional derivative. To establish sufficient decrease, it is therefore necessary to relate the exact Clarke directional derivative to the approximate directional model induced by $\bv_{k+1}$. This is achieved through the directional consistency estimate introduced below, which bounds the discrepancy between the exact subdifferential model and the approximate discrete-gradient model along the search direction.

The following result is a direct Euclidean analogue of
\cite[Proposition~3.1]{Neumaier2024}. Its proof follows the same algebraic
argument after identifying \(g_I\) with \(\bv_{k+1}\), \(q_I\) with
\(\bd_{k+1}^{\MCS}\), \(B_{II}\) with the identity matrix, and \(\delta\)
with \(\varrho\).

\begin{thm}[Angle enforcement and descent property of the \texttt{MatCSG} direction]
\label{thm.angle-enforce}
Let $\bv_{k+1} \neq 0$ be defined by
\eqref{eq:prob_vk_goldstein}, where $\mathcal V(\bx_{k+1})$ is the discrete-gradient sample defined by \eqref{eq:defmatVk}. Let $\bd_{k+1}^{\MCS}\in\R^n\setminus\{0\}$ be a search direction and define
\begin{equation}\label{eq:omega12}
\omega_{1,k+1} := \|\bv_{k+1}\|^2, \qquad
\omega_{2,k+1} := \|\bd_{k+1}^{\MCS}\|^2, \qquad
\omega_{k+1} := \bv_{k+1}^\top \bd_{k+1}^{\MCS}.
\end{equation}

Moreover, let $\varrho\in(0,1)$ and define
\[
c_{k+1} := \frac{\omega_{k+1}}{\sqrt{\omega_{1,k+1} \omega_{2,k+1}}} \in [-1,1], 
\qquad
w_{k+1} := \frac{\omega_{1,k+1} \omega_{2,k+1} (1 - c_{k+1}^2)}{1 - \varrho^2} \ge 0.
\]

If $\omega_{k+1} \le -\varrho \sqrt{\omega_{1,k+1} \omega_{2,k+1}}$, then $\bd_{k+1}^{\MCS}$
already satisfies the bounded angle condition and no correction is needed.
Otherwise, the corrected direction
\begin{equation}\label{e.dcorr}
\wt\bd_{k+1}^{\MCS} := \bd_{k+1}^{\MCS} - \varpi_{k+1} \bv_{k+1}
\end{equation}
is defined, where
\begin{equation}\label{e.lambda}
\varpi_{k+1} :=
\frac{\omega_{k+1} + \varrho \sqrt{w_{k+1}}}{\omega_{1,k+1}}.
\end{equation}

Then, the corrected direction satisfies the exact bounded angle condition
\begin{equation}\label{e.angle}
\frac{\bv_{k+1}^\top \wt\bd_{k+1}^{\MCS}}
{\|\bv_{k+1}\|\,\|\wt\bd_{k+1}^{\MCS}\|}
= -\varrho < 0,
\end{equation}
and therefore is a descent direction relative to the approximate first-order model induced by $\bv_{k+1}$. Moreover, since
\[
\bd_{k+1}^{\MCS}
=
-\bv_{k+1}+\boldsymbol{\ss}_k\bd_k^{\MCS},
\]
the corrected direction $\wt\bd_{k+1}^{\MCS}$ may be interpreted as a
safeguarded version of the memory-enhanced direction
$\bd_{k+1}^{\MCS}$, obtained by shifting it toward the approximate descent
direction $-\bv_{k+1}$ until the prescribed bounded angle condition holds.
\end{thm}

\section{Two-Point Descent Procedure via Discrete Gradients}
\label{subsecTPLS}

Throughout this section, $k$ denotes the iteration counter of the main
loop of Algorithm~\ref{alg:DG-MatCSG}, which is introduced in Section \ref{sec_stationarity_algorithm}. At each outer iteration $k$, Algorithm~\ref{alg:DG-MatCSG}
invokes the two-point \texttt{LiS} procedure (Algorithm~\ref{AlgTPLSDG}, below), which is developed in the present section. Thus, within Algorithm~\ref{AlgTPLSDG}, the index $i$ is used for the
inner \texttt{LiS} iterations, while $k$ remains fixed and refers to the
current outer iteration.

The bounded angle condition established in
Theorem~\ref{thm.angle-enforce} ensures that the corrected direction
$\wt\bd_k^{\MCS}$ satisfies $\bv_k^\top \wt\bd_k^{\MCS} < 0$, so that $\wt\bd_k^{\MCS}$ is descent-oriented with respect to the
approximate subgradient model generated by $\bv_k$.

Before introducing the two-point \texttt{LiS} procedure, we record a
sufficient condition under which the bounded-angle corrected direction
admits an Armijo step. This result clarifies the role of the directional
model error assumption (A5). It is not needed for the finite termination of
Algorithm~\ref{AlgTPLSDG}, which is proved later under the weak upper
semismoothness and directional consistency assumption (A4).

\begin{prop}[Armijo-type decrease for the corrected \texttt{MatCSG} direction]
\label{prop:MatCSG-armijo-final}
Assume that (A1) and (A5) hold. Let
$\wt\bd_{k+1}^{\MCS}\in\mathbb{R}^n\setminus\{0\}$ and
$\bv_{k+1}\neq0$ satisfy the bounded angle condition
\[
\bv_{k+1}^\top \wt\bd_{k+1}^{\MCS}
\le
-\varrho
\|\bv_{k+1}\|
\,
\|\wt\bd_{k+1}^{\MCS}\|
<0,
\]
for some $\varrho\in(0,1)$. Then, for every
$\mu_1\in(0,1-\theta)$, there exists $t_k>0$ such that
\[
f(\bx_{k+1} + t_k \wt\bd_{k+1}^{\MCS})
\le
f(\bx_{k+1})
+
\mu_1\, t_k\, \bv_{k+1}^\top \wt\bd_{k+1}^{\MCS}.
\]
\end{prop}
\begin{proof}
For brevity, set $\bx:=\bx_{k+1}$, $\bd:=\wt\bd_{k+1}^{\MCS}$, $\bv:=\bv_{k+1}$, and $\xi:=\xi_{k+1}$. By (A5), applied with $\bx=\bx_{k+1}$, $\bv=\bv_{k+1}$, and $\bd=\wt\bd_{k+1}^{\MCS}$, the directional model error conditions
\eqref{eq:dir_error_bound} and \eqref{eq:dir_error_small} hold. Hence,
by \eqref{eq:dir_error_small} and the bounded angle condition, we have
\[
\xi\|\bd\|
\le
\theta\varrho\|\bv\|\,\|\bd\|
\le
-\theta\,\bv^\top\bd.
\]
Consequently, using the directional error bound \eqref{eq:dir_error_bound}, we get
\[
f^\circ(\bx;\bd)\le
\bv^\top\bd+\xi\|\bd\|
\le
\bv^\top\bd-\theta\bv^\top\bd
=
(1-\theta)\bv^\top\bd
<0.
\]
Choosing any $\mu_1\in(0,1-\theta)$ and since $\bv^\top\bd<0$, we get
\[
f^\circ(\bx;\bd)\le
(1-\theta)\bv^\top\bd
<
\mu_1\bv^\top\bd,
\]
so that $f^\circ(\bx;\bd)
<
\mu_1\bv^\top\bd$. By the definition of the Clarke directional derivative,
\[
\limsup_{t\downarrow0}
\frac{
f(\bx+t\bd)-f(\bx)
}{t}
\le
f^\circ(\bx;\bd).
\]
Therefore, there exists $\bar t>0$ such that, for all $t\in(0,\bar t]$,
\[
\frac{
f(\bx+t\bd)-f(\bx)
}{t}
\le
f^\circ(\bx;\bd)
+
\frac{\gamma}{2}
<
\mu_1\bv^\top\bd,
\]
where $\gamma
:=
\mu_1\bv^\top\bd
-
f^\circ(\bx;\bd)
>0$. Multiplying by $t>0$ gives
\[
f(\bx+t\bd)
<
f(\bx)
+
\mu_1 t\,\bv^\top\bd.
\]
Taking any $t\in(0,\bar t]$ proves the claim.
\end{proof}

\begin{rem}[On the error condition for Armijo decrease]
\label{rem:error-condition}
The directional model error condition \eqref{eq:dir_error_bound} together with \eqref{eq:dir_error_small} ensures that the directional approximation error is sufficiently small
relative to the predicted descent along the corrected direction
$\wt\bd_{k+1}^{\MCS}$. Indeed, in view of the bounded angle condition \eqref{e.angle}, we have
\[
|\bv_{k+1}^\top \wt\bd_{k+1}^{\MCS}|
\ge
\varrho
\|\bv_{k+1}\|
\,
\|\wt\bd_{k+1}^{\MCS}\|.
\]
Therefore, \eqref{eq:dir_error_small} implies $\xi_{k+1}\|\wt\bd_{k+1}^{\MCS}\|
\le
\theta
|\bv_{k+1}^\top \wt\bd_{k+1}^{\MCS}|$. Thus, the true Clarke directional derivative remains negative and is
controlled by the predicted model decrease.

The admissible directional error level scales proportionally with
$\|\bv_{k+1}\|$. In particular, as $\|\bv_{k+1}\|\to0$, the allowable
directional error must also vanish, ensuring that the model accuracy
increases near stationary points.

From a computational perspective, this requirement is consistent with the
discrete-gradient framework, since the directional model error can be
reduced by refining the discretization parameter, enriching the direction
set, or improving the finite sampled model.

Moreover, within the two-point \texttt{LiS} procedure, failure of the
sufficient decrease condition triggers the generation of additional
discrete gradients. This potentially improves the directional approximation
quality and facilitates satisfaction of the above condition.
\end{rem}

\begin{cor}[Descent implies Armijo admissibility]
\label{cor:descent-armijo}
Under the assumptions of Theorem~\ref{thm.angle-enforce}, let $\bv_{k+1}$
be defined by \eqref{eq:prob_vk_goldstein}. Assume that the directional
model error satisfies \eqref{eq:dir_error_bound}, where $\xi_{k+1}\ge0$. Suppose that there exists $\theta\in(0,1)$ such
that \eqref{eq:dir_error_small} holds. Then, for every $\mu_1\in(0,1-\theta)$, the corrected direction
$\wt\bd_{k+1}^{\MCS}$ admits a step size $t_k>0$ such that
\[
f(\bx_{k+1} + t_k \wt\bd_{k+1}^{\MCS})
\le
f(\bx_{k+1})
+
\mu_1\, t_k\, \bv_{k+1}^\top \wt\bd_{k+1}^{\MCS}.
\]
\end{cor}

\begin{proof}
By Theorem~\ref{thm.angle-enforce}, the corrected direction
$\wt\bd_{k+1}^{\MCS}$ satisfies the bounded angle condition \eqref{e.angle}. The directional model error bound \eqref{eq:dir_error_bound} and the relative directional model-error condition \eqref{eq:dir_error_small} are precisely the remaining assumptions of
Proposition~\ref{prop:MatCSG-armijo-final}. Therefore, the claimed
Armijo-type decrease follows directly from
Proposition~\ref{prop:MatCSG-armijo-final}.
\end{proof}

The descent subgradient method of Maleknia and Soleimani--Damaneh
\cite{Maleknia2024} employs a two-point Mifflin \texttt{LiS} procedure, called \texttt{DG-TPLiS}, that either produces a serious step satisfying a sufficient decrease
condition or generates a new subgradient that enriches the approximation
of the Goldstein $\varepsilon$-subdifferential. The key feature of that
procedure is the use of two step lengths:
\begin{itemize}
\item $t_i$, which is kept inside $(0,\varepsilon)$ and is used to generate
new subgradient information;
\item $\bar t_i$, which is used to test whether a serious step satisfying
sufficient decrease can be accepted.
\end{itemize}

In the present derivative-free setting, exact subgradients are unavailable.
Therefore, they are replaced by approximate subgradients computed by the
discrete-gradient construction. Moreover, instead of using the normalized
minimal-norm subgradient direction, the search direction is generated by
the \texttt{MatCSG} formula \eqref{dkCSG} and then corrected by
\eqref{e.dcorr}--\eqref{e.lambda} to enforce a bounded angle condition.

The \texttt{DG-TPLiS} mechanism below follows the two-point structure of
\cite{Maleknia2024}. The only modification in the interval update is that
the next trial point $t_{i+1}$ is selected as the geometric mean of the two
contracted endpoints. This is a particular admissible interior choice and
does not affect the interval contraction argument.

The sufficient decrease condition used in the \texttt{DG-TPLiS} algorithm is
\[
f(\bx+t\wt\bd_k^{\MCS})-f(\bx)
\le
\mu_1 t\,\bv_k^\top \wt\bd_k^{\MCS},
\]
where $\mu_1\in(0,1)$ and
$\bv_k^\top \wt\bd_k^{\MCS}<0$. The enrichment condition is
\[
\bg_i^\top\wt\bd_k^{\MCS}
\ge
\mu_2\bv_k^\top\wt\bd_k^{\MCS},
\qquad
0<\mu_1<\mu_2<1.
\]
Since $\bv_k^\top\wt\bd_k^{\MCS}<0$, the enrichment condition identifies
a discrete gradient that is less aligned with the current descent model
than $\bv_k$ and therefore provides new information for improving the
convex approximation of the Goldstein subdifferential.

\paragraph{Normalization and practical scaling.}
\label{par:direction_normalization_scaling}
Throughout Algorithm~\ref{AlgTPLSDG}, the corrected \texttt{MatCSG}
direction is normalized before it is passed to the \texttt{DG-TPLiS}. More
precisely, after the correction step \eqref{e.dcorr}--\eqref{e.lambda}, we
use the direction
\begin{equation}\label{eq:normalized_MatCSG_direction}
\wt\bd_k^{\MCS}
:=
\frac{\tilde\bd_k^{\MCS}}{\|\tilde\bd_k^{\MCS}\|}.
\end{equation}
Since the bounded angle condition is invariant under positive scaling,
\eqref{eq:normalized_MatCSG_direction} preserves the descent orientation:
\[
\bv_k^\top \wt\bd_k^{\MCS}
\le
-\varrho\|\bv_k\|.
\]

For practical implementations, one may instead use a positively scaled
direction
\begin{equation}\label{eq:scaled_MatCSG_direction}
\widehat\bd_k^{\MCS}
:=
\alpha\,\wt\bd_k^{\MCS},
\qquad
\alpha>0,
\end{equation}
where $\alpha$ is fixed independently of $k$. This scaling does not change
the angle, since
\[
\frac{\bv_k^\top \widehat\bd_k^{\MCS}}
{\|\bv_k\|\,\|\widehat\bd_k^{\MCS}\|}
=
\frac{\bv_k^\top \wt\bd_k^{\MCS}}
{\|\bv_k\|\,\|\wt\bd_k^{\MCS}\|}.
\]
Moreover, if the \texttt{DG-TPLiS} accepts only steps satisfying
\begin{equation}\label{eq:successful_step_lower_bound}
t_k\ge\underline t>0,
\end{equation} then the accepted displacement satisfies
\[
t_k\|\widehat\bd_k^{\MCS}\|
=
\alpha t_k
\ge
\alpha\underline t.
\]
Thus, in the scaled implementation, every successful step has displacement
bounded below by the positive constant $\alpha\underline t$. Consequently,
the finite-termination arguments remain valid with $\underline\tau=\alpha\underline t$. More generally, the same conclusion holds for variable scalings $\alpha_k$
provided that
\[
0<\alpha_{\min}\le\alpha_k\le\alpha_{\max}<\infty.
\]
In the theoretical statements below, we use the normalized convention
\eqref{eq:normalized_MatCSG_direction}; the scaled version
\eqref{eq:scaled_MatCSG_direction} or its diagonal version may be used in implementation; see Subsection \ref{sec:diag}.

\begin{algorithm}[H]
\caption{Two-Point Line Search with Discrete Gradients ({\tt DG-TPLiS})}
\label{AlgTPLSDG}
\begin{algorithmic}[1]

\Require Radius $\varepsilon\in(0,1)$, current point $\bx_k\in\R^n$, base vector
$\bv_k\in\Rz^n$, corrected search direction $\wt\bd_k^{\MCS}$ satisfying
the normalization convention \eqref{eq:normalized_MatCSG_direction},
parameters $\mu_1,\mu_2$ satisfying $0<\mu_1<\mu_2<1$,
contraction parameter $\zeta\in(0,\tfrac12)$,
lower successful-step bound $\underline t\in(0,\varepsilon)$,
and integer $p\in\mathbb N$.

\Ensure A pair $(\mathfrak{s}_k,I_k)$, where $I_k=1$ denotes a successful step and
$I_k=0$ denotes a valid enrichment discrete gradient.

\State Choose $t_0\in(\underline t,\varepsilon)$ and $\bd^{(1)}\in D$. Then, set
$\lambda_0:=\lambda_k$ and $z_0:=z_k$.

\State Compute $i_{\max}\in \D\argmax_{j=1,\dots,n}|d^{(1)}_j|$
and the discrete gradient 
\[\bg_0
:=
\bg_{i_{\max}}
(\bx_k+t_0\wt\bd_k^{\MCS},\bd^{(1)},\ve,z_0,\lambda_0,\omega).
\]

\State Set $\bar t_0:=1$, $t_0^l:=0$, $t_0^u:=\varepsilon$, and $i:=0$.

\While{true}

\If{$f(\bx_k+t_i\wt\bd_k^{\MCS})-f(\bx_k)
\le
\mu_1 t_i\,\bv_k^\top\wt\bd_k^{\MCS}$}, $t_{i+1}^l:=t_i$, \quad $t_{i+1}^u:=t_i^u$.

\Else, $t_{i+1}^l:=t_i^l$, \quad $t_{i+1}^u:=t_i$.

\EndIf

\If{$f(\bx_k+\bar t_i\wt\bd_k^{\MCS})-f(\bx_k)
\le
\mu_1 \bar t_i\,\bv_k^\top\wt\bd_k^{\MCS}$
\textbf{and} $\bar t_i\ge\underline t$}

\State Set $\mathfrak{s}_k:=\bx_k+\bar t_i\wt\bd_k^{\MCS}$, $I:=1$, and \Return $(\mathfrak{s}_k,I_k)$.

\EndIf

\If{$\bg_i^\top\wt\bd_k^{\MCS}
\ge
\mu_2\bv_k^\top\wt\bd_k^{\MCS}$}

\State Set $\mathfrak{s}_k:=\bg_i$, $I:=0$, and \Return $(\mathfrak{s}_k,I_k)$.

\EndIf

\State Choose
\[
t_{i+1}
=
\sqrt{
\bigl(t_{i+1}^l+\zeta(t_{i+1}^u-t_{i+1}^l)\bigr)
\bigl(t_{i+1}^u-\zeta(t_{i+1}^u-t_{i+1}^l)\bigr)
}.
\]

\State Select direction $\bd^{(i+1)}\in D$ and update
\[
\bar t_{i+1}:=
\left[\exp\left(\frac{\log t_0}{p}\right)\right]^{i+1}.
\]

\State Set $\lambda_{i+1}:=\lambda_k$, and $z_{i+1}:=z_k$. Then, compute $i_{\max}\in \D\argmax_{j=1,\dots,n}|d^{(i+1)}_j|$
\State and the new discrete gradient
\[
\bg_{i+1}
:=
\bg_{i_{\max}}
(\bx_k+t_{i+1}\wt\bd_k^{\MCS},\bd^{(i+1)},\ve,z_{i+1},\lambda_{i+1},\omega).
\]

\State $i:=i+1$.

\EndWhile

\end{algorithmic}
\end{algorithm}

The procedure returns a pair $(\mathfrak{s},I)$. If $I=1$, then
$\mathfrak{s}$ is the new point $\mathfrak{s}=\bx_k+\bar t_i\wt\bd_k^{\MCS}$,  and it satisfies both the sufficient decrease condition and the lower
successful-step condition $\bar t_i\ge\underline t$. By the normalization convention \eqref{eq:normalized_MatCSG_direction},
every successful step satisfies \eqref{eq:successful_step_lower_bound}. Thus, the positive lower displacement used later in the finite-termination
analysis becomes an algorithmic property.

If $I=0$, then $\mathfrak{s}$ is the discrete gradient $\bg_i$ computed at
the trial point $\bx_k+t_i\wt\bd_k^{\MCS}$ using the sampling direction
$\bd^{(i)}$. In this case, $\bg_i$ is added to the current set
$\mathcal V(\bx_k)$ in order to enrich the approximation of the Goldstein
$\varepsilon$-subdifferential.

The parameter $\varepsilon$ controls the interval used to generate
discrete-gradient information through the sequence $\{t_i\}$. The auxiliary
sequence $\{\bar t_i\}$ is used for the serious-step acceptance test. Since
$\varepsilon\in(0,1)$ and $t_0\in(\underline t,\varepsilon)$, the sequence
\[
\bar t_{i+1}
=
\left[\exp\left(\frac{\log t_0}{p}\right)\right]^{i+1}
\]
is well defined and satisfies $\bar t_p=t_0$. Since $0<t_0<1$, the auxiliary sequence may be written as $\bar t_i=t_0^{i/p}$, and in particular $\bar t_p=t_0$.

The next results establish finite termination of the two-point procedure.
Unlike Proposition~\ref{prop:MatCSG-armijo-final}, they do not use the
directional model error assumption (A5); instead, they rely on the
interval-contraction mechanism and the directional consistency condition
(A4).

\begin{lem}[Interval contraction and limit properties]
\label{lem_dgtpls}
Assume that (A1) holds and suppose that Algorithm~\ref{AlgTPLSDG}
is applied with $\zeta\in(0,\tfrac12)$ and $\varepsilon\in(0,1)$. Assume that the algorithm does not terminate. Define $\Delta_i:=t_i^u-t_i^l$. Then the following assertions hold.

\begin{enumerate}
\item[(i)] For every $i\ge0$, one has $t_i\in\{t_{i+1}^l,t_{i+1}^u\}$. Moreover, for every $i\ge1$,
\[
0<\Delta_{i+1}\le(1-\zeta)\Delta_i \quad \text{and}\quad 0\le t_i^l\le t_{i+1}^l<t_{i+1}^u\le t_i^u\le\varepsilon.
\]

\item[(ii)] There exists $t^\ast\in[0,\varepsilon]$ such that $t_i^u\downarrow t^\ast$, $t_i^l\uparrow t^\ast$, and $t_i\to t^\ast$. Furthermore,
\[
t^\ast\in
T:=
\left\{
t\ge0:
f(\bx_k+t\wt\bd_k^{\MCS})-f(\bx_k)
\le
\mu_1 t\,\bv_k^\top\wt\bd_k^{\MCS}
\right\}.
\]

\item[(iii)] The index set $\mathcal I
:=
\{\,i\in\mathbb N_0:t_{i+1}^u=t_i\,\}$
is infinite.
\end{enumerate}
\end{lem}

\begin{proof}
\textbf{(i)}
The update rule in Algorithm~\ref{AlgTPLSDG} gives either
\[
t_{i+1}^l=t_i,\qquad t_{i+1}^u=t_i^u, \quad \text{or} \quad
t_{i+1}^l=t_i^l,\qquad t_{i+1}^u=t_i.
\]
Hence $t_i\in\{t_{i+1}^l,t_{i+1}^u\}$ and $[t_{i+1}^l,t_{i+1}^u]\subseteq[t_i^l,t_i^u]$. For $i\ge1$, the point $t_i$ was generated by the geometric-mean rule
from the contracted interval
\[
\left[
t_i^l+\zeta(t_i^u-t_i^l),
\;
t_i^u-\zeta(t_i^u-t_i^l)
\right].
\]
Thus, $t_i-t_i^l\ge\zeta\Delta_i$ and $
t_i^u-t_i\ge\zeta\Delta_i$. If the lower endpoint is updated, then
\[
\Delta_{i+1}=t_i^u-t_i\le(1-\zeta)\Delta_i.
\]
If the upper endpoint is updated, then $\Delta_{i+1}=t_i-t_i^l\le(1-\zeta)\Delta_i$. Hence, $0<\Delta_{i+1}\le(1-\zeta)\Delta_i$ for $i\ge1$. The nesting inequalities follow directly from the update rule and the fact
that the algorithm does not terminate.

\medskip
\noindent
\textbf{(ii)}
Since $\Delta_{i+1}\le(1-\zeta)\Delta_i$ for all $i\ge1$ and
$0<1-\zeta<1$, we have $\Delta_i\to0$. The sequence $\{t_i^l\}$ is nondecreasing and bounded above by
$\varepsilon$, while $\{t_i^u\}$ is nonincreasing and bounded below by
$0$. Hence there exist $t_l^\ast,t_u^\ast\in[0,\varepsilon]$ such that $t_i^l\uparrow t_l^\ast$ and $t_i^u\downarrow t_u^\ast$. Because $\Delta_i=t_i^u-t_i^l\to0$, we have $t_l^\ast=t_u^\ast=:t^\ast$. Since $t_i\in[t_i^l,t_i^u]$, it follows that $t_i\to t^\ast$.

We now prove that $t^\ast\in T$. We first show that $t_i^l\in T$ for all
$i\ge0$. Since $t_0^l=0$, we have $t_0^l\in T$. Suppose $t_i^l\in T$.
If the sufficient decrease condition at $t_i$ holds, then
$t_{i+1}^l=t_i$, so $t_{i+1}^l\in T$. Otherwise,
$t_{i+1}^l=t_i^l$, and again $t_{i+1}^l\in T$. Hence $t_i^l\in T$ for
all $i$. Since $f$ is locally Lipschitz, the function $t\mapsto f(\bx_k+t\wt\bd_k^{\MCS})$ is continuous. Passing to the limit along $t_i^l\uparrow t^\ast$ gives
\[
f(\bx_k+t^\ast\wt\bd_k^{\MCS})-f(\bx_k)
\le
\mu_1 t^\ast\bv_k^\top\wt\bd_k^{\MCS}.
\]
Thus $t^\ast\in T$.

\medskip
\noindent
\textbf{(iii)}
Let $\mathcal I:=\{\,i\in\mathbb N_0:t_{i+1}^u=t_i\,\}$. We first prove that $\mathcal I\neq\emptyset$. Suppose, by contradiction,
that $\mathcal I=\emptyset$. Then the upper endpoint is never updated,
and therefore the sufficient decrease condition at $t_i$ holds for all
$i\ge0$. In particular,
\[
f(\bx_k+t_0\wt\bd_k^{\MCS})-f(\bx_k)
\le
\mu_1 t_0\bv_k^\top\wt\bd_k^{\MCS}.
\]
By the definition of the auxiliary sequence,
\[
\bar t_p
=
\left[\exp\left(\frac{\log t_0}{p}\right)\right]^p
=
t_0.
\]
Since $t_0>\underline t$, the successful-step test is satisfied at
iteration $p$, which contradicts the assumption that the algorithm does
not terminate. Hence $\mathcal I\neq\emptyset$.

Now suppose that $\mathcal I$ is finite. Since upper endpoint updates occur
only finitely many times, there exists $\bar i$ such that $t_i^u$ is
constant for all $i\ge\bar i$. By part (ii), $t_i^u\downarrow t^\ast$; therefore this constant must be $t^\ast$, and hence $t_i^u=t^\ast$, for all $i\ge\bar i$. Because $\mathcal I\neq\emptyset$ and is finite, let $j:=\max\mathcal I$. Then, by the definition of $\mathcal I$, $t_{j+1}^u=t_j$. Since no upper endpoint updates occur after index $j$, the sequence
$\{t_i^u\}$ is constant for all $i\ge j+1$. By part (ii), we also have $t_i^u\downarrow t^\ast$. Therefore, this constant value must be $t^\ast$, and hence $t_j=t_{j+1}^u=t^\ast$.
Moreover, because the upper endpoint was updated at index $j$, the
sufficient decrease condition failed at $t_j$, namely,
\[
f(\bx_k+t_j\wt\bd_k^{\MCS})-f(\bx_k)
>
\mu_1t_j\bv_k^\top\wt\bd_k^{\MCS}.
\]
Since $t_j=t^\ast$, this contradicts $t^\ast\in T$. Therefore
$\mathcal I$ must be infinite.
\end{proof}

The previous lemma is purely algorithmic: it describes the contraction of
the trial interval and the limiting behavior of the two-point \texttt{DG-TPLiS}
sequence under the assumption of nontermination. We now use this structure,
together with the directional consistency condition in (A4), to prove finite
termination of Algorithm~\ref{AlgTPLSDG}. Notice that the Armijo
admissibility result in Proposition~\ref{prop:MatCSG-armijo-final} is not
used in this proof. That proposition gives a sufficient local condition for
existence of an Armijo step under the directional model error assumption
(A5), whereas the finite termination result below relies instead on the
enrichment mechanism and the weak upper semismoothness/directional
consistency condition (A4).

\begin{thm}[Finite termination of {\tt DG-TPLiS}]
\label{thm_dgtpls_final_clean}
Suppose that (A1) and (A4) hold.
Assume that Algorithm~\ref{AlgTPLSDG} is applied with
\[
\varepsilon\in(0,1),
\qquad
\underline t\in(0,\varepsilon),
\qquad
p\in\mathbb N,
\qquad
0<\mu_1<\mu_2<1.
\]
Let $\bd:=\wt\bd_k^{\MCS}$, $
q:=\bv_k^\top\bd<0$, $I=I_k$, $\mathfrak{s}=\mathfrak{s}_k$, $\bx=\bx_k$, and $\bv=\bv_k$ Then, Algorithm~\ref{AlgTPLSDG} terminates after finitely many inner iterations. More precisely, it returns either
\[
I=1 \quad \text{and}
\quad
\mathfrak{s}=\bx+\bar t_i\bd \quad \text{with}\quad f(\bx+\bar t_i\bd)-f(\bx)
\le
\mu_1\bar t_i\bv^\top\bd,
\qquad
\bar t_i\ge\underline t,
\]
or it returns $I=0$ and $\mathfrak{s}=\bg_i$ with $\bg_i^\top\bd
\ge
\mu_2\bv^\top\bd$.
\end{thm}

\begin{proof}
Assume, by contradiction, that Algorithm~\ref{AlgTPLSDG} does not
terminate. Let $\mathcal I$ be the infinite index set from
Lemma~\ref{lem_dgtpls}(iii). For every $i\in\mathcal I$, the upper
endpoint is updated, and hence the sufficient decrease condition at
$t_i$ fails:
\[
f(\bx+t_i\bd)-f(\bx)
>
\mu_1t_i\bv^\top\bd.
\]
By Lemma~\ref{lem_dgtpls}(ii), $t_i\to t^\ast$ and $t^\ast\in T$, so $f(\bx_k+t^\ast\bd)-f(\bx_k)
\le
\mu_1t^\ast\bv^\top\bd$. Combining the last two inequalities gives, for all $i\in\mathcal I$,
\[
f(\bx+t_i\bd)-f(\bx+t^\ast\bd)
>
\mu_1(t_i-t^\ast)\bv^\top\bd.
\]
Since $i\in\mathcal I$ implies $t_i=t_{i+1}^u$, and since
$t_{i+1}^u>t^\ast$ for every finite nonterminal index, we have $t_i>t^\ast$. Define $h_i:=t_i-t^\ast>0$ and $z:=\bx+t^\ast\bd$. Then $h_i\downarrow0$ along $\mathcal I$, and
\[
\frac{f(z+h_i\bd)-f(z)}{h_i}
>
\mu_1\bv^\top\bd
=
\mu_1 q,
\qquad
i\in\mathcal I.
\]
Therefore, $\D\liminf_{i\in\mathcal I,\ i\to\infty}
\frac{f(z+h_i\bd)-f(z)}{h_i}
\ge
\mu_1 q$. By the directional consistency condition \eqref{ass:wus_dg_lis} in (A4), 
\begin{equation}\label{eq:dgtplis_limsup_gtd_bound}
\limsup_{i\in\mathcal I,\ i\to\infty}\bg_i^\top\bd
\ge
\mu_1 q.
\end{equation}
On the other hand, since the algorithm does not terminate through the
enrichment condition, we have $\bg_i^\top\bd
<
\mu_2\bv^\top\bd
=
\mu_2q$ for all $i$. Hence, $\displaystyle\limsup_{i\in\mathcal I,\ i\to\infty}\bg_i^\top\bd
\le
\mu_2q$. Since $q<0$ and $0<\mu_1<\mu_2<1$, we have $\mu_2q<\mu_1q$. Thus,
\[
\mu_1q
\overset{\eqref{eq:dgtplis_limsup_gtd_bound}}{\le}
\limsup_{i\in\mathcal I,\ i\to\infty}\bg_i^\top\bd
\le
\mu_2q
<
\mu_1q,
\]
which is impossible. Therefore, Algorithm~\ref{AlgTPLSDG} must terminate
after finitely many inner iterations.
\end{proof}

Theorem~\ref{thm_dgtpls_final_clean} establishes finite termination of the
inner \texttt{DG-TPLiS} procedure. Finite termination of the outer
{\tt DG-MatCSG} loop requires an additional argument controlling the decrease
of the minimal-norm element of the convex hull after enrichment.

\section{Computation of an Approximate Goldstein Stationary Point}
\label{sec_stationarity_algorithm}

In this section, we adapt the descent subgradient framework of
Maleknia and Soleimani--Damaneh \cite{Maleknia2024} to the discrete
gradient setting. The objective is to compute a
$(\delta,\mathcal V(\bx_k))$ approximate Goldstein stationary point in the
sense of \eqref{eq:goldstein_stationarity_dg} using only function
evaluations.

At each iteration $k$, we maintain a finite set $\mathcal V(\bx_k) \subset \mathbb{R}^n$, consisting of discrete gradients computed at $\bx_k$ and along trial
points generated by the \texttt{DG-TPLiS} procedure.

We distinguish between:
\begin{itemize}
\item the \bfi{base discrete gradient} $\hat{\bv}_k \in \mathcal V(\bx_k)$ used in the \texttt{LiS};
\item the \bfi{minimal-norm element} $\bv_k$ defined by \eqref{eq:prob_vk_goldstein}.
\end{itemize}
By construction, $\|\bv_k\| \le \|\hat{\bv}_k\|$. The complete version of the method is summarized in
Algorithm~\ref{alg:DG-MatCSG}. Starting from an initial point $\bx_0$,
a discrete gradient $\hat{\bv}_0$ is computed and used to initialize the set
$\mathcal V(\bx_0)$, and we set $\mathcal V(\bx_0):=\{\hat{\bv}_0\}$. At iteration $k$, we compute $\bv_k$ by \eqref{eq:prob_vk_goldstein}. If
$\|\bv_k\| \le \delta$, then, by Definition \ref{eq:goldstein_stationarity_dg}, the point
$\bx_k$ is a $(\delta,\mathcal V(\bx_k))$ approximate Goldstein stationary point, and the procedure terminates.

Moreover, under the approximation properties
\eqref{eq:outer_dg}--\eqref{eq:approx_g}, Proposition~\ref{prop:DG_to_Goldstein}
implies that $\bx_k$ is also an approximate Goldstein stationary point
with respect to $\partial_\varepsilon f(\bx_k)$.

Otherwise, a search direction is computed using the {\tt MatCSG}
construction \eqref{dkCSG}. Since the \texttt{LiS} uses the scalar product
$\hat{\bv}_k^\top \bd_k$, the correction step
\eqref{e.dcorr}--\eqref{e.lambda} is applied with respect to the current
base vector $\hat{\bv}_k$ in order to ensure descent with respect to
$\hat{\bv}_k$.

A modified two-point {\tt LiS} procedure, {\tt DG-TPLiS}, is
then applied along $\bd_k$ using $\hat{\bv}_k$. This procedure either
produces a trial point satisfying a sufficient decrease condition or returns an additional discrete gradient that enriches the set
$\mathcal V(\bx_k)$.

If {\tt DG-TPLiS} produces a new point $\mathfrak{s}_k$ $(I_k=1)$, we set $\bx_{k+1} = \mathfrak{s}_k$. A new base discrete gradient $\hat{\bv}_{k+1}$ is computed at
$\bx_{k+1}$, and the new set is initialized as $\mathcal V(\bx_{k+1}) = \{\hat{\bv}_{k+1}\}$. The minimal-norm vector $\bv_{k+1}$ is then computed from this set at the
next iteration.

Otherwise $(I_k=0)$, the current point remains unchanged and the returned
vector is added to the set $\mathcal V(\bx_k)$, i.e., $\bx_{k+1} = \bx_k$, $\mathcal V(\bx_{k+1})
=
\mathcal V(\bx_k)\cup\{\mathfrak{s}_k\}$, and $\hat{\bv}_{k+1} := \mathfrak{s}_k$. The iteration counter is then increased and the procedure repeats.

During enrichment iterations $(I_k=0)$, the base point remains unchanged
and the search direction is not recomputed. Instead, the current direction
is reused while the set $\mathcal V(\bx_k)$ is enriched by additional
discrete gradients. The reused direction is assumed to preserve the bounded angle condition
with respect to all enrichment vectors generated while the base point
remains unchanged.

The enrichment vectors returned by Algorithm~\ref{AlgTPLSDG}
are interpreted as approximate subgradients associated with the current
base point through the local approximation properties of the discrete
gradient construction. More precisely, they are computed at trial points
associated with the fixed base point and the current search direction.

In Algorithm~\ref{AlgTPLSDG}, the symbol $\bv_k$ denotes the base vector
passed to \texttt{DG-TPLiS}. When the procedure is invoked inside
Algorithm~\ref{alg:DG-MatCSG}, this base vector is instantiated as
$\hat{\bv}_k$, not necessarily as the minimal-norm element of
$\conv(\mathcal V(\bx_k))$.

\begin{algorithm}[H]
\caption{Discrete-Gradient Descent Method with {\tt MatCSG} Direction ({\tt DG-MatCSG})}
\label{alg:DG-MatCSG}

\textbf{Input:}
Initial point $\bx_0$, tolerance $\delta>0$, angle condition parameter $\varrho\in(0,1)$.\\
\textbf{Output:} A $(\delta,\mathcal V(\bx))$ approximate Goldstein stationary point

\textbf{Initialization:}
Compute an initial discrete gradient $\hat{\bv}_0$ at $\bx_0$ and set $\mathcal V(\bx_0) = \{\hat{\bv}_0\}$, and $k=0$.

\begin{algorithmic}[1]

\While{true}

\State Compute $\bv_k$ by \eqref{eq:prob_vk_goldstein}.

\State {\bf if} $\|\bv_k\| \le \delta$ {\bf then}, \Return $\bx_k$; {\bf end if}

\If{$k=0$ or $I_{k-1}=1$}
\State Compute the raw \texttt{MatCSG} direction $\bd_k^{\MCS}$ by
\eqref{dkCSG}.
\State Apply the correction \eqref{e.dcorr}--\eqref{e.lambda} with respect
to $\hat{\bv}_k$.
\State Then, normalize it according to
\eqref{eq:normalized_MatCSG_direction} as $\|\bd_k\|=1$ and
$\hat{\bv}_k^\top\bd_k\le-\varrho\|\hat{\bv}_k\|$.
\Else
\State Set $\bd_k:=\bd_{k-1}$.
\EndIf

\State Apply Algorithm~\ref{AlgTPLSDG} with base vector $\hat{\bv}_k$
to obtain $(\mathfrak{s}_k,I_k)$.

\If{$I_k = 1$}

\State Set $\bx_{k+1} = \mathfrak{s}_k$, compute $\hat{\bv}_{k+1}$ at $\bx_{k+1}$, and set $\mathcal V(\bx_{k+1}) = \{\hat{\bv}_{k+1}\}$.

\Else

\State Set $\bx_{k+1} = \bx_k$, $\mathcal V(\bx_{k+1}) = \mathcal V(\bx_k)\cup\{\mathfrak{s}_k\}$, and $\hat{\bv}_{k+1} := \mathfrak{s}_k$.

\EndIf

\State $k \leftarrow k+1$.

\EndWhile

\end{algorithmic}
\end{algorithm}

Because Algorithm~\ref{AlgTPLSDG} is used with a positive lower
successful-step bound, every successful step satisfies \eqref{eq:successful_step_lower_bound}, i.e., $t_k=\bar t_i\ge \underline t>0$. Moreover, by the normalization convention
\eqref{eq:normalized_MatCSG_direction}, the \texttt{LiS} direction used in
Algorithm~\ref{alg:DG-MatCSG} satisfies $\|\bd_k\|=1$. Hence, the algorithmic lower displacement property becomes $t_k\|\bd_k\|\ge \underline t$. Thus, no separate lower-displacement assumption is needed. The following lemma is the discrete-gradient analogue of the successful-step finiteness
argument used in the descent subgradient framework.

\begin{lem}[Finiteness of successful iterations in DG--MatCSG]
\label{lem:A_finite_matcsg_strong}
Assume that (A1) and (A3) hold. Suppose that the \texttt{LiS} directions
satisfy the normalization convention \eqref{eq:normalized_MatCSG_direction}, i.e., $\|\bd_k\|=1$ holds and that every successful step returned by Algorithm~\ref{AlgTPLSDG}
satisfies $t_k\ge \underline t>0$. Assume that the bounded angle condition holds at every successful
iteration, namely,
\[
\hat{\bv}_k^\top \bd_k
\le
-\varrho \|\hat{\bv}_k\|\,\|\bd_k\|,
\qquad
\varrho\in(0,1).
\]
Let $A:=\{k\in\mathbb N_0:I_k=1\}$. If Algorithm~\ref{alg:DG-MatCSG} does not terminate, then $A$ is finite.
\end{lem}

\begin{proof}
Assume by contradiction that $A$ is infinite. Since the algorithm does not
terminate, the stopping criterion fails for every $k$, and hence $\|\bv_k\|>\delta$ for all $k$. For each $k\in A$, Algorithm~\ref{AlgTPLSDG} returns a successful step, so the sufficient decrease condition gives
\[
f(\bx_{k+1})-f(\bx_k)
\le
\mu_1 t_k\,\hat{\bv}_k^\top\bd_k.
\]
Using the bounded angle condition and the normalization
$\|\bd_k\|=1$, we obtain
\[
\hat{\bv}_k^\top\bd_k
\le
-\varrho\|\hat{\bv}_k\|.
\]
Moreover, by the successful-step lower bound, $t_k\ge\underline t>0$. Therefore, for every $k\in A$,
\[
f(\bx_{k+1})-f(\bx_k)
\le
-\mu_1\varrho\,\underline t\,\|\hat{\bv}_k\|.
\]
Since $\hat{\bv}_k\in\mathcal V(\bx_k)\subset\conv(\mathcal V(\bx_k))$ and $\bv_k$ is the minimal-norm element of
$\conv(\mathcal V(\bx_k))$, we have $\|\bv_k\|\le\|\hat{\bv}_k\|$. Hence, because the algorithm has not terminated, $\|\hat{\bv}_k\|\ge\|\bv_k\|>\delta$. Consequently, for every $k\in A$, $f(\bx_{k+1})-f(\bx_k)
\le
-\mu_1\varrho\,\underline t\,\delta$. Thus, each successful iteration decreases the objective value by at least
the fixed positive amount $\mu_1\varrho\,\underline t\,\delta>0$. For enrichment iterations, the base point is unchanged, and therefore $f(\bx_{k+1})=f(\bx_k)$. If $A$ were infinite, then the objective values would satisfy $f(\bx_k)\to-\infty$. On the other hand, all iterates remain in the level set $\mathcal L(\bx_0)$. By (A3), this level set is bounded. Since $f$ is locally Lipschitz by
(A1), it is continuous; hence $\mathcal L(\bx_0)$ is closed and compact.
Therefore $f$ is bounded below on $\mathcal L(\bx_0)$, which contradicts
$f(\bx_k)\to-\infty$. Therefore $A$ must be finite.
\end{proof}

\begin{thm}[Finite termination of \texttt{DG-MatCSG}]
\label{thm:finite_termination_matcsg_final}
Assume that (A1) and (A3) hold, and suppose that the hypotheses of
Lemma~\ref{lem:A_finite_matcsg_strong} hold. Assume further that after finitely many successful iterations, whenever
the base point remains fixed and only enrichment iterations occur, the
search direction is reused. Denote this fixed direction by $\bar d$.
Suppose that the bounded angle condition is preserved along the enrichment
sequence, namely,
\[
\hat{\bv}_k^\top \bar d
\le
-\varrho\|\hat{\bv}_k\|\,\|\bar d\|,
\qquad
\forall k\in\mathbb  N_0,
\]
during consecutive enrichment iterations at the fixed base point.
Assume also that $\|\bar d\|=1$. Then Algorithm~\ref{alg:DG-MatCSG} terminates in a finite number of
iterations.
\end{thm}

\begin{proof}
Assume by contradiction that Algorithm~\ref{alg:DG-MatCSG} does not
terminate. Then the stopping criterion fails for every $k$, and hence $\|\bv_k\|>\delta$ for all $k$.

By Lemma~\ref{lem:A_finite_matcsg_strong}, the set of successful
iterations $A:=\{k\in\mathbb N_0:I_k=1\}$ is finite. Therefore, there exists $\bar k$ such that for all
$k\ge\bar k$, only enrichment iterations occur. Hence the base point is
fixed: $\bx_k=\bar\bx$ for all $k\ge\bar k$. During these consecutive enrichment iterations, the direction is reused;
denote it by $\bar d$.

Since $I_k=0$ for all $k\ge\bar k$, Algorithm~\ref{AlgTPLSDG} returns an
enrichment discrete gradient $\mathfrak{s}_k$ satisfying
\[
\mathfrak{s}_k^\top \bar d
\ge
\mu_2\hat{\bv}_k^\top\bar d.
\]
By Algorithm~\ref{alg:DG-MatCSG}, this returned vector becomes the next
base discrete gradient: $\hat{\bv}_{k+1}:=\mathfrak{s}_k$. Therefore,
\begin{equation}\label{eq:enrichment_recursive_angle}
\hat{\bv}_{k+1}^\top\bar d
\ge
\mu_2\hat{\bv}_k^\top\bar d.
\end{equation}
Defining $a_k:=\hat{\bv}_k^\top\bar d$, using the preserved bounded angle condition, and $\|\bar d\|=1$, we obtain
\[
a_k
=
\hat{\bv}_k^\top\bar d
\le
-\varrho\|\hat{\bv}_k\|
<0.
\]
Thus $a_k<0$ for all $k\ge\bar k$. From
\eqref{eq:enrichment_recursive_angle}, we get
\begin{equation}\label{eq:ak_recursive_bound}
a_{k+1}\ge \mu_2 a_k.
\end{equation}
Since $0<\mu_2<1$ and $a_k<0$, we have $\mu_2 a_k>a_k$; therefore, $a_{k+1}>a_k$. Hence $\{a_k\}_{k\ge\bar k}$ is strictly increasing. Since $a_k<0$, it is
bounded above by $0$. Therefore, there exists $\ell\le0$ such that $a_k\to\ell$. Passing to the limit in \eqref{eq:ak_recursive_bound} gives $\ell\ge\mu_2\ell$. Equivalently, $(1-\mu_2)\ell\ge0$. Because $1-\mu_2>0$ and $\ell\le0$, it follows that $\ell=0$. Thus $a_k\to0$. Using again the preserved bounded angle condition and $\|\bar d\|=1$, we
obtain $|a_k|
=
-\hat{\bv}_k^\top\bar d
\ge
\varrho\|\hat{\bv}_k\|$. Since $a_k\to0$, it follows that $\|\hat{\bv}_k\|\to0$. Moreover, $\hat{\bv}_k\in\mathcal V(\bx_k)\subset\conv(\mathcal V(\bx_k))$. Since $\bv_k$ is the minimal-norm element of
$\conv(\mathcal V(\bx_k))$, we have $\|\bv_k\|\le\|\hat{\bv}_k\|$, so that $\|\bv_k\|\to0$. This contradicts the condition $\|\bv_k\|>\delta$ for all $k$. Therefore, Algorithm~\ref{alg:DG-MatCSG} must terminate after finitely many iterations.
\end{proof}

\section{Computation of an Approximate Clarke Stationary Point}
\label{sec:clarke_stationarity}

The objective of this section is to compute a Clarke stationary point
by solving a sequence of approximate stationarity problems. To this end, we consider sequences $\{\delta_\nu\}_{\nu\in\Nz_0}$ and
$\{\varepsilon_\nu\}_{\nu\in\Nz_0}$ such that $\delta_\nu \downarrow 0$ and $
\varepsilon_\nu \downarrow 0$. At each outer iteration $\nu$, we compute a
$(\delta_\nu,\mathcal V(\bx_{\nu+1}))$ approximate Goldstein stationary point
in the sense of \eqref{eq:goldstein_stationarity_dg} by applying
Algorithm~\ref{alg:DG-MatCSG}. That is,
\[
\bx_{\nu+1} = \texttt{DG\text{-}MatCSG}(\bx_\nu,\varepsilon_\nu,\delta_\nu)
\]
satisfies $\min_{\bv \in \conv(\mathcal V(\bx_{\nu+1}))} \|\bv\|
\le \delta_\nu$. By Proposition~\ref{prop:DG_to_Goldstein}, this implies that
$\bx_{\nu+1}$ is a
$(\delta_\nu + \bar\delta_\nu + \eta_{\nu},
\mathcal G_{\varepsilon_\nu}(\bx_{\nu+1}))$
approximate Goldstein stationary point, namely,
\[
\min_{\bu \in \conv(\mathcal G_{\varepsilon_\nu}(\bx_{\nu+1}))} \|\bu\|
\le
\delta_\nu + \bar\delta_\nu + \eta_{\nu},
\]
where $\bar\delta_\nu$ and $\eta_{\nu}$ denote the approximation errors
appearing in \eqref{eq:outer_dg} and \eqref{eq:approx_g}.

The quantities $\bar\delta_\nu$ and $\eta_\nu$ are not explicitly computed
within the algorithm. They arise from the approximation properties of
the discrete-gradient construction and are assumed to vanish asymptotically
as the sampling becomes sufficiently dense. The overall procedure is given in Algorithm \ref{alg:clarke} below. We refer to Algorithm~\ref{alg:clarke} as \texttt{DG\text{-}Clarke},
short for the Discrete-Gradient Clarke-Stationarity algorithm.

\begin{algorithm}[H]
\caption{\texttt{DG\text{-}Clarke}: Discrete-Gradient Clarke-Stationarity Algorithm}
\label{alg:clarke}

\textbf{Input:}
Starting point $\bx_0 \in \mathbb{R}^n$, sequences $\{\delta_\nu\}\downarrow 0$,
$\{\varepsilon_\nu\}\downarrow 0$, and tolerance $\eta > 0$.

\textbf{Output:}
An approximation of a Clarke stationary point.

\begin{algorithmic}[1]

\State Initialize $\nu := 0$

\While{true}

\State $\bx_{\nu+1} := \texttt{DG\text{-}MatCSG}(\bx_\nu,\varepsilon_\nu,\delta_\nu)$

\State {\bf if} $\delta_\nu \le \eta$ \textbf{and} $\varepsilon_\nu \le \eta$ {\bf then}, \Return $\bx_{\nu+1}$; {\bf end if}

\State $\nu \leftarrow \nu + 1$

\EndWhile

\end{algorithmic}
\end{algorithm}

For the convergence analysis below, we consider the idealized infinite
version of Algorithm~\ref{alg:clarke}, obtained by omitting the finite
stopping test involving $\eta$. The stopping test is used only for
practical computation.

Assumption (A2) enters the following convergence result through the
finite-sample consistency inclusion for the discrete-gradient model.
Therefore, the theorem is stated directly in terms of the inclusion
\begin{equation}\label{eq:dg_consistency_outer}
\conv(\mathcal V(\bx_{\nu+1}))
\subset
\partial_{\varepsilon_\nu} f(\bx_{\nu+1})+B(0,\bar\delta_\nu),
\qquad
\bar\delta_\nu\to0.
\end{equation}
Once \eqref{eq:dg_consistency_outer} holds, the remaining argument is
purely variational and relies only on the local boundedness, convexity,
and outer semicontinuity of the Clarke subdifferential. We now establish the asymptotic convergence of this procedure.

\begin{thm}[Convergence to a Clarke stationary point]
\label{thm:clarke_convergence}
Assume that (A1)--(A3) hold and that the discrete-gradient consistency
inclusion \eqref{eq:dg_consistency_outer} holds with $\bar\delta_\nu\to0$.
Suppose that the idealized infinite version of Algorithm~\ref{alg:clarke}
generates points $\bx_{\nu+1}$ satisfying
\[
\min_{\bv\in\conv(\mathcal V(\bx_{\nu+1}))}\|\bv\|
\le
\delta_\nu,
\]
where $\delta_\nu\to0$ and $\varepsilon_\nu\to0$. Assume moreover that the generated sequence remains in the initial level
set, that is, $\bx_\nu\in\mathcal L(\bx_0)$ for all $\nu\ge0$. Then any accumulation point $\bx^\ast$ of the shifted sequence
$\{\bx_{\nu+1}\}_{\nu\in\Nz_0}$ generated by
Algorithm~\ref{alg:clarke} is a Clarke stationary point of \(f\), i.e.,
\(0\in\partial f(\bx^\ast)\). Consequently, the same conclusion holds for
the accumulation points of \(\{\bx_\nu\}_{\nu\in\Nz_0}\).
\end{thm}

\begin{proof}
Since $\bx_\nu\in\mathcal L(\bx_0)$ for all $\nu\ge0$ and
$\mathcal L(\bx_0)$ is bounded by (A3), the sequence $\{\bx_\nu\}_{\nu\in\Nz_0}$ admits
accumulation points. Let $\bx^\ast$ be an accumulation point of the shifted
sequence $\{\bx_{\nu+1}\}_{\nu\in\Nz_0}$, and let $\bx_{\nu_j+1}\to\bx^\ast$. By construction, for each $\nu_j$ there exists
\[
\bv_{\nu_j+1}\in\conv(\mathcal V(\bx_{\nu_j+1}))
\]
such that $\|\bv_{\nu_j+1}\|\le\delta_{\nu_j}$. Using the consistency inclusion \eqref{eq:dg_consistency_outer}, there
exists
\[
\bw_{\nu_j+1}\in
\partial_{\varepsilon_{\nu_j}}f(\bx_{\nu_j+1})
\]
such that $\|\bv_{\nu_j+1}-\bw_{\nu_j+1}\|\le\bar\delta_{\nu_j}$. Hence
\[
\|\bw_{\nu_j+1}\|
\le
\|\bv_{\nu_j+1}\|
+
\|\bv_{\nu_j+1}-\bw_{\nu_j+1}\|
\le
\delta_{\nu_j}+\bar\delta_{\nu_j}
\to0.
\]
Since $\bw_{\nu_j+1}\in\partial_{\varepsilon_{\nu_j}}f(\bx_{\nu_j+1})$, by the definition of the Goldstein subdifferential,
\[
\bw_{\nu_j+1}
\in
\cl\conv
\left(
\bigcup_{\|\by-\bx_{\nu_j+1}\|\le\varepsilon_{\nu_j}}
\partial f(\by)
\right).
\]
Using the closure in the definition of
$\partial_{\varepsilon_{\nu_j}}f(\bx_{\nu_j+1})$, for each $j$ we may
choose vectors
\[
\wt{\bw}_{\nu_j+1}\in
\conv
\left(
\bigcup_{\|\by-\bx_{\nu_j+1}\|\le\varepsilon_{\nu_j}}
\partial f(\by)
\right)
\]
such that $\|\wt{\bw}_{\nu_j+1}-\bw_{\nu_j+1}\|\le 1/j$. Then
\[
\|\wt{\bw}_{\nu_j+1}\|
\le
\|\bw_{\nu_j+1}\|
+
\|\wt{\bw}_{\nu_j+1}-\bw_{\nu_j+1}\|
\to0.
\]
By Carathéodory's theorem, there exist coefficients
\[
\lambda_i^{(j)}\ge0,
\qquad
\sum_{i=0}^n\lambda_i^{(j)}=1,
\]
points $\by_i^{(j)}$ with $\|\by_i^{(j)}-\bx_{\nu_j+1}\|\le\varepsilon_{\nu_j}$, and subgradients $\bg_i^{(j)}\in\partial f(\by_i^{(j)})$ such that
\[
\wt{\bw}_{\nu_j+1}
=
\sum_{i=0}^n
\lambda_i^{(j)}\bg_i^{(j)}.
\]
Since $\bx_{\nu_j+1}\to\bx^\ast$ as
$\varepsilon_{\nu_j}\to0$, we have $\by_i^{(j)}\to\bx^\ast$ for each $i=0,\dots,n$. By local boundedness of the Clarke subdifferential near $\bx^\ast$, the
sequences $\{\bg_i^{(j)}\}_j$ are bounded. Passing to a subsequence if
necessary, we may assume
\[
\bg_i^{(j)}\to\bg_i^\ast,
\quad \text{and} \quad\lambda_i^{(j)}\to\lambda_i^\ast,
\]
with $\lambda_i^\ast\ge0$ and
$\sum_{i=0}^n\lambda_i^\ast=1$. By outer semicontinuity of the Clarke subdifferential,
\[
\bg_i^\ast\in\partial f(\bx^\ast),
\qquad
i=0,\dots,n.
\]
Passing to the limit in $\wt{\bw}_{\nu_j+1}
=
\sum_{i=0}^n
\lambda_i^{(j)}\bg_i^{(j)}$, and using $\|\wt{\bw}_{\nu_j+1}\|\to0$, we obtain $0
=
\sum_{i=0}^n
\lambda_i^\ast\bg_i^\ast$. Since $\partial f(\bx^\ast)$ is convex, the convex combination on the
right-hand side belongs to $\partial f(\bx^\ast)$. Hence, $0\in\partial f(\bx^\ast)$; this proves the claim.
\end{proof}

\section{Numerical Results}\label{sec:numerical_results}

In this section, we evaluate the performance of three variants of our algorithm \texttt{DG\text{-}Clarke} on a standard set of benchmark problems. These variants use, respectively, the steepest-descent direction, the approximate conjugate subgradient direction, and the approximate matrix-conjugate subgradient direction, all of which are normalized.  We compare them with three existing solvers: {\tt DDG-Bundle} \cite{Karmitsa2016DDGBundle}, {\tt LDGB} \cite{KarmitsaBagirov2012LDGBM,KarmitsaBagirov2011LDGBM}, and {\tt DGM} \cite{BagirovKarasozenSezer2008DGM}. To perform the experiments in MATLAB, we developed MEX interfaces for the available Fortran implementations of these solvers. The corresponding Fortran codes are available from Napsu~Karmitsa's software page:
\url{https://www.napsu.karmitsa.fi/}, under the directories
{\tt dgm}, {\tt ldgbm}, and {\tt ddgbundle}.

The supplementary material {\tt suppMat.pdf} \cite{suppMat} provides the details of the
numerical test environment and the implementation settings. It describes the
finite-max nonsmooth test environment {\tt TEminmax}, which is stored in the
accompanying MATLAB file {\tt TE.mat}. This file contains the problem
structures, including the problem name, dimension, number of finite-max terms,
loss type, data set, starting points, bound information, and objective-function
handle. In the current version, {\tt TEminmax} contains \(80\) problems:
\(50\) real-data problems generated from five data sets and ten finite-max loss
functions, together with \(30\) additional large-scale synthetic finite-max
problems. The dimensions range from \(n=8\) to \(n=1000\), while the number of
finite-max terms ranges from \(m=200\) to \(m=4177\). 

For each problem, the corresponding hitlist file stores reference information
used only for benchmarking. This includes the best available reference point,
the reference objective value, bound information, and related accuracy
diagnostics. The value \(f_{\opt}\) used in the stopping test (see \eqref{e.qs}, below) is taken from the
hitlist as the best available reference value. Therefore, \(f_{\opt}\) is not
assumed to be known in a practical run of the algorithm; it is used only to
define the benchmarking criterion and to ensure that all solvers are stopped by
the same external rule.

The supplementary material also reports the fixed tuning parameters and
computational safeguards used in the implementation, including the
discrete-gradient, line-search, bundle, scaling, and matrix-stability settings.
These details are placed in the supplementary material to keep the main
numerical section focused on the performance results. In particular, the same
parameter policy is used across the proposed variants, so the comparisons are
not driven by problem-dependent tuning.

The files {\tt suppMat.pdf}, {\tt TEminmax}, {\tt TE.mat}, and the hitlists for
all \(80\) finite-max problems are available at the GitHub repository
{\tt DG-Clarke-TEminmax}:

\url{https://github.com/GS1400/DG-Clarke-TEminmax}.

The {\tt DG-Clarke} solver is also publicly available at

\url{https://github.com/GS1400/DG_Clarke}.

\subsection{Stopping criteria}

To measure the progress of each solver \(s\in\mathcal S\), we use the relative objective reduction
\begin{equation}\label{e.qs}
q_s :=
\frac{f_s-f_{\opt}}{f_0-f_{\opt}},
\end{equation}
where \(\mathcal S\) is the set of solvers being compared. Here, \(f_s\) denotes the best objective value produced by solver \(s\), \(f_0\) is the objective value at the common initial point, and \(f_{\opt}\) is the best known objective value for the problem. The value \(f_{\opt}\) is typically obtained from the best point found by a collection of local and global derivative-free methods and proposed solvers in the present paper, and may correspond to a global minimizer or to a high-quality local minimizer. Thus, \(q_s\) is used only as a benchmarking metric; in practical applications, \(f_{\opt}\) is usually not available.

Solver \(s\) is said to solve a problem if
\[
    q_s \leq \varepsilon
\]
before either of the prescribed computational limits is reached, namely the maximum number of function evaluations \(\texttt{nfmax}\) or the maximum CPU time \(\texttt{secmax}\). If this condition is not satisfied within these limits, the problem is classified as unsolved for that solver.

The parameters \(\varepsilon\), \(\texttt{nfmax}\), and \(\texttt{secmax}\) are selected to make the comparison informative: they are chosen so that the strongest solver solves at least half of the test set, except when the noise level is sufficiently large that increasing the evaluation or time budget no longer yields a meaningful improvement in robustness or efficiency. In the numerical experiments, we used
\[
    \texttt{secmax}=600~\text{sec},\quad
    \texttt{nfmax}=2\times 10^6,\quad
    \varepsilon\in\{10^{-6},10^{-4},10^{-2}\},\quad \text{and} \quad
   2\leq n\leq 10^3.
\]

\subsection{Diagonal scaling from previously accepted best points}\label{sec:diag}

We use a bounded diagonal scaling vector computed from recently obtained
best points. Let $\bx_k^{(1)},\bx_k^{(2)},\ldots,\bx_k^{(r_k)}$ with $r_k\le m_s$ denote the stored best points, ordered according to their objective values,
with \(\bx_k^{(1)}\) being the best one. Here \(m_s\) is the prescribed
memory size used for the scaling archive; $m_s=3$ is used for the current comparison. If fewer than two best points are
available, we set \(\bs_k=\mathbf 1\). Otherwise, for each coordinate
\(j=1,\ldots,n\), we compute
\[
\bar s_{k,j}
=
\D\median_{\ell=2,\ldots,r_k}
\left|
x_{k,j}^{(\ell)}-x_{k,j}^{(1)}
\right|.
\]
Nonfinite and zero entries are replaced by one. The scaling vector is then
clipped componentwise as
\[
s_{k,j}
=
\min\left\{
s_{\max},
\max\left\{
s_{\min},\bar s_{k,j}
\right\}
\right\},
\qquad j=1,\ldots,n,
\]
where $0<s_{\min}\le s_{\max}<\infty$. Thus, with \(S_k=\operatorname{diag}(\bs_k)\), the matrix \(S_k\) is uniformly
bounded and uniformly nonsingular.

In the implementation, after computing the \texttt{MatCSG} direction
\(\bd_k^{\MCS}\) by \eqref{dkCSG}, we apply the diagonal scaling before the
bounded-angle correction. More precisely, \(S_k\bd_k^{\MCS}\) is used as the
input direction in the correction step \eqref{e.dcorr}--\eqref{e.lambda}.
The resulting corrected direction is then normalized according to
\eqref{eq:normalized_MatCSG_direction} before being passed to the two-point
line-search procedure.

This order is essential. Unlike scalar positive scaling, diagonal scaling can
change the angle between the search direction and \(\bv_k\). Therefore, the
scaling is applied before the bounded-angle correction, while the final
direction used in the line search still satisfies the bounded-angle condition
required in the analysis. Consequently, the line-search and finite-termination
arguments remain unchanged.

\subsection{Results for $n=$2--1000}
\label{resall}

Figure~\ref{f.f2} reports performance profiles with respect to the number of
function evaluations. Recall that $\rho(\tau)$ gives the fraction of problems
solved within a factor $\tau$ of the best solver on each problem. Thus, values
near $\tau=1$ reflect efficiency on the easiest instances, whereas the plateau
for large $\tau$ reflects robustness, i.e., the fraction of problems solved
within the prescribed budgets.

For the low-accuracy tolerance $\varepsilon=10^{-2}$, shown in the left panel,
both proposed variants attain the highest final values of $\rho(\tau)$ and
therefore solve the largest fraction of problems. The benchmark solvers are
competitive only over part of the range of $\tau$, but their profiles plateau at
lower levels.

For the medium-accuracy tolerance $\varepsilon=10^{-4}$, shown in the middle
panel, the same trend remains visible. The proposed variants again achieve the
largest solved fractions, while {\tt DGM}, {\tt DDGBUNDLE}, and {\tt LDGBM}
plateau earlier and at lower values. This indicates that the proposed methods
are more robust over the full test set.

For the high-accuracy tolerance $\varepsilon=10^{-6}$, shown in the right panel,
the advantage of the matrix variant becomes more pronounced.
{\tt DG-Clarke-MatCSGdir} solves a substantially larger fraction of problems
than all other solvers. In contrast, {\tt DG-Clarke-SDdir} and {\tt DGM}
plateau at much lower values, while {\tt DDGBUNDLE} and {\tt LDGBM} solve very
few problems under this stricter requirement.

Overall, the profiles show a clear robustness advantage for the proposed
methods, especially for tighter tolerances. The steepest-descent variant is
competitive for lower and medium accuracy, but its robustness deteriorates at
$\varepsilon=10^{-6}$. The matrix conjugate-subgradient variant
{\tt DG-Clarke-MatCSGdir} remains the most reliable solver across the three
accuracy levels.

\begin{figure}[H]
\centering
\scalebox{0.5}{%
\includegraphics[width=10cm,height=10cm]{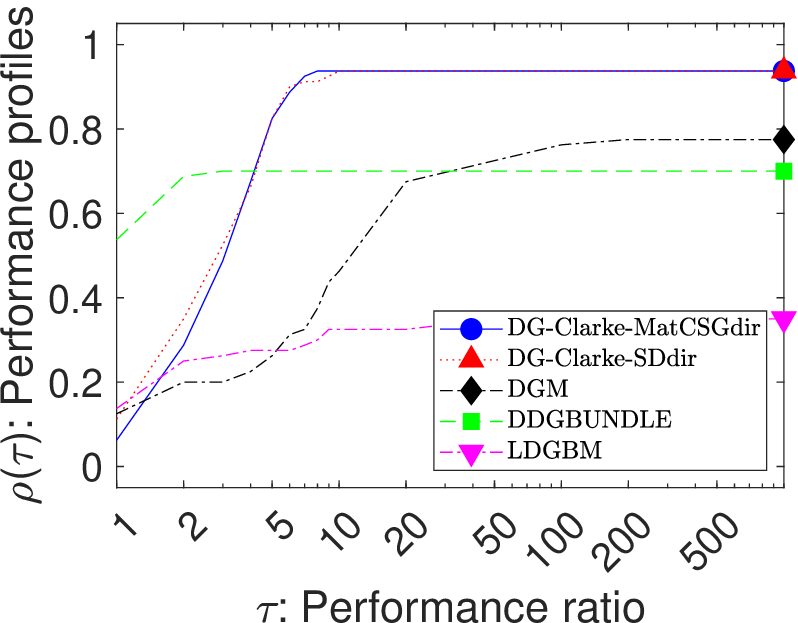}
\hspace{0.5cm}
\includegraphics[width=10cm,height=10cm]{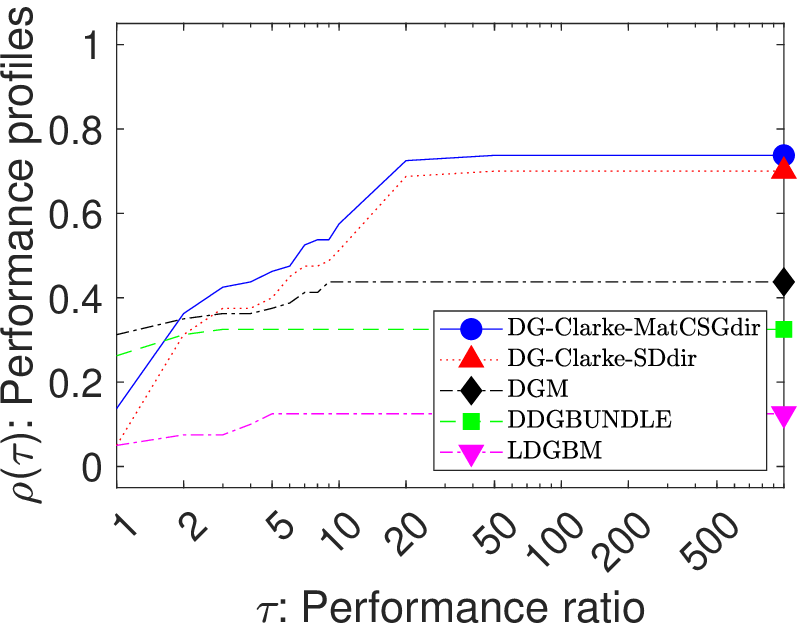}
\hspace{0.5cm}
\includegraphics[width=10cm,height=10cm]{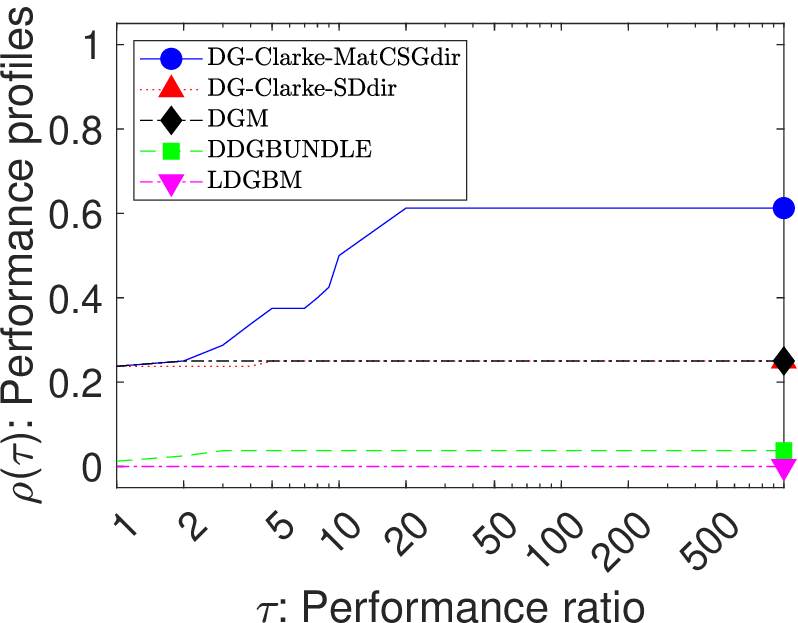}}
\caption{Performance profiles of two variants of our solver and the three
state-of-the-art solvers based on the number of function evaluations
(\texttt{nf}) with ${\tt nfmax}=2\times 10^6$ and ${\tt sexmax}=600$ sec:
$\varepsilon=10^{-2}$ (left, low accuracy), $\varepsilon=10^{-4}$ (middle,
medium accuracy), and $\varepsilon=10^{-6}$ (right, high accuracy).}
\label{f.f2}
\end{figure}

\section{Conclusions}\label{sec:conclusion}

In this paper, we proposed a derivative-free matrix conjugate-subgradient framework for unconstrained nonsmooth optimization of locally Lipschitz functions. The method uses discrete gradients to construct a finite sampled model of the Goldstein subdifferential and computes a minimal-norm element of the convex hull of these sampled vectors. This vector provides both a practical stationarity measure and a reference direction for constructing descent-oriented search directions. The main algorithmic novelty is the introduction of a derivative-free matrix conjugate-subgradient direction. 

This direction extends the basic discrete-gradient steepest-descent direction by incorporating memory from previous iterations through a structured matrix correction. Since the available first-order information is only approximate, several safeguards were introduced, including coefficient damping, dominance control, diagonal scaling, matrix-stability checks, and bounded-angle correction. These mechanisms make the memory-enhanced direction more stable and robust in a fully derivative-free nonsmooth setting.

We also adapted exact-subgradient sampling and enrichment ideas to the approximate-subgradient setting by replacing exact subgradients with discrete-gradient approximations computed only from function values. The two-point line-search procedure was designed to either accept a serious step or enrich the local sampled model, thereby improving robustness without requiring analytical derivatives or active-index information. Under the stated consistency and directional assumptions, the finite sampled model is consistent with Goldstein-type stationarity, and the safeguarded search directions satisfy the required descent-oriented properties. 

The numerical experiments demonstrate that the proposed variants are more robust
than the benchmark DFON solvers on the tested problem
collections. For the lower and medium accuracy requirements, both proposed
variants solve the largest fraction of problems within the prescribed
function-evaluation and time budgets. Under the strictest tolerance,
$\varepsilon=10^{-6}$, the advantage of the matrix conjugate-subgradient variant
becomes more pronounced: {\tt DG-Clarke-MatCSGdir} remains the most reliable
solver, whereas the steepest-descent variant and the benchmark solvers plateau
at substantially lower solved fractions. Overall, the results show that the
matrix memory mechanism improves robustness as the accuracy requirement becomes
more demanding.

Future work will focus on reducing the cost of discrete-gradient sampling for large-scale problems, developing adaptive rules for locality and sampling parameters, and extending the proposed framework to constrained nonsmooth optimization problems.

\section*{Declarations}

\noindent
\textbf{Data Availability:} The data that support the findings of this research can be obtained from the corresponding author, subject to a reasonable request.\\

\noindent
\textbf{Conflict of Interest:} No competing interests are declared by the authors.\\

\noindent
\textbf{Ethical Statement:} It is declared that this research did not involve any studies with human participants or animals, and that there are no ethical issues associated with this work.

{\bf Funding} Morteza Kimiaei acknowledges financial support of the Austrian Science Foundation under \url{https://doi.org/10.55776/PAT2747625}.


\bibliographystyle{plain}
\bibliography{References}


\end{document}